\documentclass{lms}

\usepackage{amsmath,amsfonts,amssymb,enumerate,mathrsfs,mathtools,tikz}
\usepackage{hyperref}
\usetikzlibrary{matrix,arrows}

\usepackage{enumitem}
\newlist{indlist}{itemize}{4}
\setlist[indlist,1]{%
    label={},
    noitemsep,
    leftmargin=0pt,
    }
\setlist[indlist]{
    label={},
    noitemsep,
    }

\newtheorem{theorem}{Theorem}[section] 
\newtheorem{lemma}[theorem]{Lemma}     
\newtheorem{cor}[theorem]{Corollary}
\newtheorem{prop}[theorem]{Proposition}

\newunnumbered{conj1}{Conjecture 1} 
\newunnumbered{conj2}{Conjecture 2}
\newnumbered{defn}[theorem]{Definition}
\newnumbered{hypothesis}{Hypothesis}
\newnumbered{rmk}[theorem]{Remark}
\newunnumbered{theorem1}{\bf Theorem 1}
\newunnumbered{cor1}{\bf Corollary 1}
\newunnumbered{factA}{Fact A}
\newnumbered{fact}[theorem]{\bf{Fact}}

\providecommand{\val}{\mathcal{O}}

\providecommand{\M}{\mathcal{M}}

\providecommand{\Q}{\mathbb{Q}_2}
\providecommand{\Z}{\mathbb{Z}}
\providecommand{\sq}{K^{\times} / (K^{\times})^2}
\providecommand{\F}{\mathbb{F}_2}

\title{Recovering $p$-adic valuations from pro-$p$ Galois groups}

\author{Jochen Koenigsmann and Kristian Strommen}

\classno{12J10, 11S20 (primary), 11G25 (secondary)}

\extraline{The second author (KS) was supported by an EPSRC Studentship at Oxford University, Award Number MATH1115. The work of this article makes up part of his PhD thesis, under supervision by the first author (JK).}

\begin{document}
\maketitle

\begin{abstract}
Let $K$ be a field with $G_K(2) \simeq G_{\Q}(2)$, where $G_F(2)$ denotes the maximal pro-2 quotient of the absolute Galois group of a field $F$. We prove that then $K$ admits a (non-trivial) valuation $v$ which is 2-henselian and has residue field $\F$. Furthermore, $v(2)$ is a minimal positive element in the value group $\Gamma_v$ and $[\Gamma_v:2\Gamma_v]=2$. This forms the first positive result on a more general conjecture about recovering $p$-adic valuations from pro-$p$ Galois groups which we formulate precisely. As an application, we show how this result can be used to easily obtain number-theoretic information, by giving an independent proof of a strong version of the birational section conjecture for smooth, complete curves $X$ over $\Q$, as well as an analogue for varieties.
\end{abstract}

\section{Introduction}
\label{sec:introduction}

One of the most fruitful philosophies of modern number theory has been that one can understand the arithmetic of a field $K$ by understanding the structure of the absolute Galois group $G_K=Gal(K^{sep}/K)$, where $K^{sep}$ is a separable closure of $K$. One instance of this is the celebrated Neukirch-Uchida Theorem, which states that, for algebraic number fields $K_1$ and $K_2$, if $G_{K_1} \simeq G_{K_2}$ then $K_1 \simeq K_2$. A similar result for $p$-adic local fields was obtained by Mochizuki in \cite{mochi_padic}, provided one adds some minimal extra structure on the Galois group.

More generally, one could ask what information about an \emph{abstract} field is determined by its absolute Galois group. It is well known (by work of Artin-Schreier et al.) that a field $K$ is real-closed if and only if $G_K \simeq G_F$, where $F$ is a real-closed field. The equivalent result for $p$-adically closed fields was obtained in \cite{koe1}. In particular, $G_K \simeq G_{\mathbb{Q}_p}$ iff $K$ is elementarily equivalent to $\mathbb{Q}_p$ in the language of rings. This was used in \cite{koe3} to give a proof of the section conjecture in birational anabelian geometry for $p$-adic fields. In both cases the idea is to show that the abstract structure of the Galois group encodes the existence of an ordering (resp. a henselian valuation), which then determines the arithmetic of the field.

Because $G_K$ is often highly complex, it is reasonable to ask how much information can be obtained when making use of smaller quotients. Of particular interest are the maximal pro-$l$ quotients $G_K(l)$ of $G_K$, $l$ a prime, since for $p$-adic fields these quotients are well understood. Indeed, if $\zeta_l$ is a primitive $l$-th root of unity, $l \not=p$, then $G_{\mathbb{Q}_p(\zeta_l)}(l) \simeq \mathbb{Z}_l \rtimes \mathbb{Z}_l$ via the action of the cyclotomic character. Conversely, it is shown in \cite{koe2} that if $K$ is a field containing $\zeta_l$ and $G_K(l) \simeq \mathbb{Z}_l \rtimes \mathbb{Z}_l$, then $K$ admits a non-trivial $l$-henselian valuation\footnote{For the definition of $l$-henselianity, $l$ a prime, see Section 2.} with residue field of characteristic different from $l$. The valuation is recovered using the theory of `rigid elements' (see e.g. \cite{ep}, Section 2.2.3) via a combinatorial argument. The existence of suitable rigid elements is in turn inferred from the structure of $G_K(l)$, which forces $K$ to satisfy a `local reciprocity law' in the form of a bijection between extensions of degree $l$ and subgroups of $K^{\times}$ of index $l$ induced by the norm map.

The case $l \not=p$ lives in the `tamely ramified' part of the Galois group. Much more mysterious is the `wild' case $l=p$. Here the structure of $G_{F}(p)$, $F$ a finite extension of $\mathbb{Q}_p$ containing $\zeta_p$, is known by work of Demushkin, Labute and Serre (cf. \cite{serre2} and \cite{labute}). It is an example of a pro-$p$ Demushkin group given by generators and relations which can be specified (see Section 2). These fields are therefore canonical examples of what Efrat has called \emph{$p$-Demushkin fields} (see Section 2.2). As already alluded to when mentioning the result of \cite{mochi_padic}, there is no hope of recovering \emph{all} the arithmetic of such a $p$-Demushkin field purely from the maximal pro-$p$ quotient of its absolute Galois group, since there exist non-isomorphic pairs $F, F'$ of finite extensions $\mathbb{Q}_p$ containing $\zeta_p$ for which $G_F \simeq G_{F'}$. Examples of such are given in Section 2 of \cite{jarden1979characterization}, and include e.g. the pairs $F=\mathbb{Q}_p(\zeta_p, \sqrt[p]{p}), F'=\mathbb{Q}_p(\zeta_p, \sqrt[p]{p+1})$, for any prime $p>2$. Further examples in the pro-$p$ context are given in Remarks \ref{rmk:diff_gal_groups} and \ref{rmk:diff_gal_groups_p2}.

With this indeterminacy in mind, we make the following conjecture:\\

\begin{conj1}
\emph{Let $F/\mathbb{Q}_p$ be a finite extension, $K$ an arbitrary field, where both $F$ and $K$ contain $\zeta_p$. Suppose $G_F(p) \simeq G_K(p)$. Then we conjecture that the characteristic of $K$ must be $0$. Furthermore, we conjecture that there exists a non-trivial valuation $v$ on $K$ such that for some finite extension $F'/\mathbb{Q}_p$ with $G_{F'}(p) \simeq G_F(p)$ and $p$-adic valuation $w$ on $F'$, the following holds:}
\begin{itemize}
\item[(1)] $v$ is \emph{$p$-henselian}
\item[(2)] $F'w = Kv$
\item[(3)] $[\Gamma_v : p\Gamma_v]=p$
\item[(4)] There is a uniformizer $\pi$ of $(K,v)$, in the sense that $v(\pi )$ is minimal positive in $\Gamma_v$
(so the valuation is discrete\footnote{Note that this is meant in the more general valuation-theoretic sense that the value group has a minimal positive element, not that the value group is isomorphic to $\Z$.}), and $K$ and $F^\prime$ have the same initial ramification index over $\mathbb{Q}$, i.e., there is a natural number $e$ with $v(p) = e\cdot v(\pi )$ and $w(p) = e\cdot w(\pi^\prime)$, where $\pi^\prime$ is a uniformizer of $F^\prime$.
\end{itemize}
\emph{Here $Kv$, $F'w$ and $\Gamma_v$ denote the residue fields and value group of the respective valuations.}
\end{conj1}

\noindent Thus conjecturally, the `wild' part of $G_K$ sees a lot more of the structure of the field than the `tame' part. In fact, in the spirit of the Elementary Type Conjecture (see e.g. the introduction of \cite{jacobware} as well as \cite{efratetc}), we expect that the following conjecture holds:\\

\begin{conj2}
\emph{Suppose $G_K(p)$ is a finitely generated pro-$p$ Demushkin group of rank $\geq 3$. Then there is a finite extension $F/\mathbb{Q}_p$ containing $\zeta_p$ such that $G_K(p) \simeq G_F(p)$.}\\
\end{conj2}

\noindent That is, we expect that essentially the \emph{only} examples of finitely generated pro-$p$ Galois groups which are Demushkin are the ones coming from finite extensions of $\mathbb{Q}_p(\zeta_p)$, and that the structure implied by being Demushkin is already enough to force the existence of a valuation which is as close to being $p$-adic as one could reasonably hope. A proof of this would be a major step forward in the programme of classifying all finitely generated pro-$p$ Galois groups (see \cite{efratetc}, particularly sections 5.2 and 5.3, for a discussion on this point).\\

\begin{rmk} (\emph{Counterexamples to sharper versions of Conjecture 1})\\
In special cases the valuation recovered in Conjecture 1 can be assumed to have nicer properties; for example, $F'$ can sometimes be taken be to $F$ itself (see Remark \ref{cool} and Proposition \ref{bsc}). However, it is in general not possible to sharpen the Conjecture considerably, as the following examples demonstrate.
\begin{itemize}
\item The valuation in Conjecture 1 cannot be taken to be rank 1, since e.g. $\mathbb{Q}_p(\!(\mathbb{Q})\!)$ has absolute Galois group isomorphic to that of $\mathbb{Q}_p$, and the associated $p$-adic valuation has higher rank. Here $\mathbb{Q}_p(\!(\mathbb{Q})\!)$ denotes the field of formal power series with coefficients in $\mathbb{Q}_p$ and powers in $\mathbb{Q}$.
\item One also cannot expect more than $p$-henselianity in Conjecture 1. For example, let $K$ be the fixed field of a decomposition subgroup of $G_{\mathbb{Q}(\zeta_p)}(p)$ with respect to any prolongation of the $p$-adic valuation on $\mathbb{Q}(\zeta_p)$. Then $G_K(p) \simeq G_{\mathbb{Q}_p(\zeta_p)}(p)$ and $K$ is $p$-henselian but not henselian.
\item The uniformizer $\pi$ of $(K,v)$ in part (4) of Conjecture 1 cannot always be chosen to be {\em algebraic} over $\mathbb{Q}$. For example, given a prime $p>2$, let $F/ \mathbb{Q}_{p}(\zeta_p)$ be a proper finite totally ramified extension
of degree prime to $p$, let $\pi\in F$ be a uniformizer of $F$ transcendental over $\mathbb{Q}$,
and let $K = F\cap [\mathbb{Q}(\zeta_p,\pi )(p)]$ be the $p$-henselisation of $\mathbb{Q}(\zeta_p,\pi )$
with respect to the $p$-adic valuation induced from $F$. Then $G_K(p)\cong G_F(p)$ as the restriction $G_F(p) \to G_K(p)$ is onto by construction and as $F$ is in fact, the completion of $K$, so, in particular $F(p)$ is the compositum of $F$ and $K(p)$. Now $\pi$ is a uniformizer of both $F$ and $K$, but no uniformizer $\pi^\prime$ of $F$ which is algebraic over $\mathbb{Q}$ could be a uniformizer of $K$, as otherwise F would be contained in the compositum of $\mathbb{Q}_p(\zeta_p)$ and the algebraic part of K, hence in $\mathbb{Q}_p(\zeta_p)(p)$
which is impossible since $F/\mathbb{Q}_p(\zeta_p)$ is a proper extension of degree prime to p.

\end{itemize}
\end{rmk}

\noindent The most significant previous work on these conjectures was done by Efrat in \cite{efrat}. He proved a conditional result that a field $K$ satisfying the conditions of Conjecture 1 must be so-called ``arithmetically Demushkin", equivalent to the existence of a valuation on $K$ of a $p$-adic nature (Theorem 6.3b of ibid). The additional assumption required was that $K$ already admitted a valuation with non-$p$-divisible value group and whose fixed field of a decomposition subgroup of $G_K(p)$ does not contain all $p^n$-th roots of unity. It was conjectured that in fact all Demushkin fields are arithmetically Demushkin. Under the same assumption, a proof of Conjecture 2 is also obtained. By comparison, Conjecture 1 here represents a sharpening of Efrat's conjecture: if $K$ admits a valuation with the properties specified in Conjecture 1, then $K$ is necessarily arithmetically Demushkin, by way of taking a suitable coarsening\footnote{Note that if $K$ is as in Conjecture 1, then $G_K(p)$ necessarily is Demushkin of rank $\geq 3$, since $G_F(p)$ has rank at least 3 when $F$ is a $p$-adic field containing $\zeta_p$. Therefore $K$ satisfies the conditions in \cite{efrat}.}. 

The main difficulty in proving unconditional versions of these conjectures is to produce, from the Galois structure alone, any non-trivial valuation whatsoever. In this paper we will prove the following unconditional result, which represents the first successful attempt at overcoming this hurdle so far:\\

\begin{theorem1}
\emph{Conjecture 1 is true in the case $F = \Q$. That is, if $K$ is any field with $G_K(2) \simeq G_{\Q}(2)$, then there exists a non-trivial $2$-henselian valuation $v$ on $K$ such that the residue field $Kv$ is $\F$, the value group $\Gamma_v$ is discrete with $v(2)$ a minimal positive element and $[\Gamma_v : 2\Gamma_v]=2$.}\\
\end{theorem1}

\noindent The proof hinges on the fact, first noted in Frohn's thesis \cite{frohn}, that any field $K$ for which $G_K(p)$ is Demushkin satisfies a `local reciprocity' law induced by the norm map (see Proposition \ref{frohn}). The proof then proceeds in three steps. First, one uses the explicit structure of $G_{\Q}(2)$ to make some preliminary observations about $\sq$, which in turn imply that $char(K)=0$. Secondly, putting $k = K \cap \overline{\mathbb{Q}}$, we show using the Albert-Brauer-Hasse-Noether local-global principle that, unless we are in an exceptional case, $k$ embeds into $\mathbb{Q}_2 \cap \overline{\mathbb{Q}}$ and hence that $\sq$ is generated as an $\F$-vector space by $-1,2$ and $5$, as is the case for $\Q$. Assuming we are not in the exceptional case, a consequence of this and of `local reciprocity' is that the `lattice' of norm subgroups $Norm(K(\sqrt{a})^{\times}) \leqslant K^{\times}$ for $a \in K \setminus K^2$ is identical to that of $\mathbb{Q}_2$. Thirdly, we use an adaptation of the rigid element method to construct a valuation ring which satisfies all the desired properties. The fact that our construction yields a valuation ring depends on checking that certain subsets in $K$ live inside specific norm subgroups. It turns out that this depends purely on the `combinatorics' of the lattice of norm subgroups arising under the operation of intersecting multiple such subgroups. Because the lattice in $K$ matches that of $k$, the existence of the desired valuation ring is in this way effectively `lifted' from $k$ to $K$. The verification of this part is carried out via explicit computations which have a distinctly K-theoretic flavour. Indeed, the technique locates elements in specific norm groups by intersecting multiple Steinberg relations, and in some sense resembles doing arithmetic combinatorics over the 2-truncated big Witt vectors of $K$. It would of course be ideal to obtain a more conceptual, less computation-heavy, proof, by making this vague statement precise, but this has so far eluded us.

Most of the calculations are relegated to the appendices to aid exposition. Finally, a proof by contradiction shows that the exceptional case from step 2 simply cannot occur, by using the same norm-combinatorics techniques.

We do not have any new ideas of how to prove Conjecture 2, so the result in \cite{efrat} (his Theorem 7.3) remains the state of the art.

It is the opinion of the authors that the general framework of the proof should be generalizable to cases where $p > 2$. If one can recover the correct norm-lattice, then similar computations should suffice to extract the required valuation. In fact, since the rigid element construction is simpler in the case $p>2$, these computations may end up being conceptually simpler. However, the methods used here to pin down the norm-lattice depended on properties specific to $p=2$ and it is therefore unclear how to proceed in the general case.

Finally, as an application, in Section 5 we show that Theorem 1 has the following corollary, which gives a significant strengthening of the birational section conjecture (see \cite{koe3})  over $\Q$.\\

\begin{cor1}\emph{Suppose $X$ is a smooth, complete curve over $K=\Q$. Then every group-theoretic section of the canonical projection
\begin{equation}
G_{K(X)}(2) \twoheadrightarrow G_K(2) \nonumber
\end{equation}
is induced by a unique rational point\footnote{See Section 5 for the precise meaning of this.} $a \in X(K)$, where $K(X)$ is the function field of $X$.}
\end{cor1}

\noindent In fact, by more closely analysing the proof of the main theorem, we show that one can get away with an even smaller quotient of the Galois group, the so-called maximal $\mathbb{Z}/p$ elementary meta-abelian quotient (see the introduction to \cite{popmeta}). Thus we recover as a consequence of our work the main result of \cite{popmeta} in the case where the base field is $\Q$. We also prove a statement for varieties, strengthening the main result of \cite{stixvarieties}.

\begin{acknowledgement}
We would like to express our deep gratitude to the anonymous referee who has made numerous extremely helpful suggestions which have greatly improved the paper.
\end{acknowledgement}

\section{Preliminaries}

We collect some technical definitions and results used. All the valuation theory needed can be found in \cite{ep}. The results on Galois cohomology and Demushkin groups are all stated and proved in \cite{serre2} (Chapter 5) and \cite{labute}. The results on Kummer theory, Brauer groups, central simple algebras and Galois cohomology over an arbitrary field are comprehensively covered in \cite{gille2017central} (Chapter 4 in particular). An alternative reference for several of these points is \cite{Neukirch2007} (Chapters 3, 6 and 7 in particular). Finally, a particularly useful reference is Frohn's thesis \cite{frohn}, which conveniently covers most of the required background.

\subsection{$p$-Henselian valuations}
Given a valuation $v$ on a field $K$, we let $Kv$ denote the residue field and $\Gamma_v$ the value group. The valuation ring will be denoted by $\val_v$.

\begin{defn}
For a field $K$ and a prime $p$, we write $K(p)$ for the {\bf maximal pro-$p$ Galois extension of $K$} (within some fixed algebraic closure of $K$), i.e., the compositum of all finite Galois extensions of $K$ of order a power of $p$. A field is said to be {\bf $p$-closed} if $K=K(p)$.

If $v$ is a valuation on $K$, we say $v$ is {\bf $p$-henselian} if $v$ has a unique extension to $K(p)$, or, equivalently, if Hensel's lemma holds for polynomials that split in $K(p)$.
\end{defn}

\begin{fact}
\label{fact:phenselian}
\emph{A field $(K,v)$ is $p$-henselian if and only if $v$ extends uniquely to every Galois extension of degree $p$.} (\cite{ep} Theorem 4.2.2)
\end{fact}

\subsection{Demushkin groups and their invariants}

\begin{defn}
Let $G$ be a pro-$p$ group for some prime $p$. $G$ is said to be {\bf Demushkin}\footnote{Note that other sources may define this slightly differently.} of rank $n>1$ ($n \in \mathbb{N}$) if
\begin{itemize}
\item[(i)] $H^1(G, \Z / p\Z)$ is a finite $\mathbb{F}_p$-vector space with  $dim_{\mathbb{F}_p}(H^1(G, \Z / p\Z))=n$.
\item[(ii)] $dim_{\mathbb{F}_p}(H^2(G, \Z / p\Z))=1$.
\item[(iii)] The cup product $H^1(G, \Z / p\Z) \times H^1(G, \Z / p\Z) \rightarrow H^2(G, \Z / p\Z)$ is a non-degenerate bilinear pairing.
\end{itemize}
\end{defn}
The first condition says that $G$ is finitely generated as a pro-$p$ group, with a minimal set of generators having size $n$. The second condition says that a minimal presentation requires precisely one relation (see \cite{serre2} Section 4.2) or Section 1 of \cite{labute}). Hence $G \simeq F/(r)$, where $F$ is the free pro-$p$ group on $n$ generators, $r$ is an element (or `word') of $F^p[F,F]$, and $(r)$ is the cloxsed normal subgroup generated by $r$. Here brackets denote the commutator. Consideration of the abelianisation $G^{ab} = G/[G,G]$ of $G$ gives the following

\begin{lemma}
\label{g_ab_lemma}
There exists a unique $s \in \mathbb{N} \cup \{\infty\}$ such that
\begin{equation}
G^{ab} \simeq  \Z_p / p^s \Z_p \times \Z_p ^{n-1} \nonumber
\end{equation}
with $n = \text{rank}(G)$ and $p^{\infty} = 0$ by convention.
\end{lemma}
\begin{proof}
Taking the abelianization kills all the terms of the word $r$ involving commutators. If $r$ involves no $p$-th powers, then the abelianization is clearly just a finitely generated pro-$p$ abelian group on $n$ generators, and we are in the case $s=\infty$. Suppose then that $r$ does contain some $p$-th powers. In this case trivialising the commutators leaves behind an expression in $G^{ab}$ of the form $g_1^{p^{s_1}}g_2^{p^{s_2}}\ldots g_k^{p^{s_k}}=1$ for some $k \leq n$. Let $s$ be the smallest of the integers $s_i$: without loss of generality $s=s_1$. Then $G^{ab}$ has the following presentation as a pro-$p$ abelian group:
\begin{equation}
G^{ab} \simeq \langle g_1, g_2, \ldots, g_n \mid w^{p^s}=1 \rangle, \nonumber
\end{equation}
where $w=g_1g_2^{p^{s_2-s}}\ldots g_k^{p^{s_k-s}} \not= 1$. The set $\{w, g_2, \ldots, g_n\}$ clearly also generates $G^{ab}$, and no new non-trivial relations are introduced besides $w^{p^s}=1$, since such non-trivial relations would induce ones among $\{g_1, \ldots, g_n\}$. Thus $G^{ab}$ has a presentation of the form
\begin{equation}
G^{ab} \simeq \langle w, g_2, \ldots, g_n \mid w^{p^s}=1 \rangle, \nonumber
\end{equation}
from which one gets
\begin{equation}
G^{ab} \simeq  \Z / p^s \Z \times \Z_p ^{n-1} \simeq  \Z_p / p^s \Z_p \times \Z_p ^{n-1}  \nonumber
\end{equation}
as desired.
\end{proof}

\noindent The invariants $n$ and $s$ of $G$ almost determine $G$ up to isomorphism. To completely determine $G$, one needs the additional invariant $\text{Im}(\chi)$ introduced by Serre (\cite{serre2} Section 4.5). Serre has shown that there exists a continuous homomorphism $\chi : G \rightarrow \Z_p^{\times}$ which is uniquely determined by the property that the induced homorphisms $H^1(G, I_j(\chi)) \rightarrow H^1(G, I_1(\chi))$ are surjective for $j \geq 1$, where $I_j(\chi)$ is the $G$-module defined by letting $G$ act on $\Z / p^j \Z$ via $\chi$; see Theorem 4 of \cite{labute} for a proof. It can then be shown that the image $\text{Im}(\chi)$ is an invariant of $G$ such that $s$ is the highest power of $p$ for which $\text{Im}(\chi) \subset 1 + p^s\Z_p$ (Corollary to Theorem 4 in \cite{labute}). Furthermore, $n$ and $\text{Im}(\chi)$ completely determine $G$ up to isomorphism (Theorem 2 of \cite{labute}). In fact, Demushkin, Serre and Labute have given a complete classification of possible Demushkin groups in terms of these invariants, which can be split into three cases C1, C2 and C3. As before, the value of $p^{\infty}$ is set to be 0 by convention.

\begin{theorem} (Classification of Demushkin groups)
\label{thm:classification}
Let $G$ be a pro-$p$ Demushkin group with invariants $n \in \mathbb{N}$, $s \in \mathbb{N} \cup \{\infty\}$ and $A=Im(\chi)$. Let $U_p^{(f)} := 1+p^f \mathbb{Z}_p$. Then one of the following cases holds:\\
\begin{itemize}
\item[C1.] If $p^s \not= 2$ then $n$ is even and the relation $r$ can be chosen as
\begin{equation}
r=x_1^{p^s}[x_1,x_2]...[x_{n-1},x_{n}]. \nonumber
\end{equation}
In this case $A = U_{p}^{(s)}$.\\
\item[C2.] If $p^s=2$ and $n$ is odd, then
\begin{equation}
r=x_1^2 x_2^{2^f}[x_2,x_3][x_4,x_5]\ldots[x_{n-1},x_{n}], \nonumber
\end{equation}
with $f$ an integer $\geq 2$ or $\infty$ such that $A = \{\pm 1\}\times U_2^{(f)}$.\\
\item[C3.] Finally, let $p^s=2$ and $n$ be even. If $A^2$ has index 2 in $A$ then
\begin{equation}
r=x_1^{2+2f}[x_1,x_2]\ldots[x_{n-1},x_{n}] \nonumber
\end{equation}
with $f$ an integer $\geq 2$ or $\infty$. In this case $A$ is the closed subgroup of $\mathbb{Z}_2^{\times}$ generated by $-1+2f$. If $A^2$ has index 4 in $A$ then
\begin{equation}
r=x_1^{2}[x_1,x_2]x_3^{2f}\ldots[x_{n-1},x_{n}] \nonumber
\end{equation}
for some integer $f \geq 2$. In this case $A = \{\pm 1\}\times U_2^{(f)}$. \\
\end{itemize}
Conversely, all the above cases of $r$ define a Demushkin group with the appropriate invariants.
\end{theorem}
\begin{proof}
Details and proofs of all cases are given in \cite{labute}, with the three cases described in the Introduction.
\end{proof}

\noindent The canonical examples of Demushkin groups are the Galois groups $G_K(p)$ for finite extensions $K$ of $\mathbb{Q}_p$ containing $\zeta_p$ (\cite{serre2}, Section 5.6).

\subsection{Demushkin fields and field-theoretic interpretation of the Demushkin invariants}
\label{sec:demushkin_fields}

As just remarked, the groups $G_K(p)$ for finite extensions $K$ of $\mathbb{Q}_p$ containing $\zeta_p$ are Demushkin. These fields are therefore canonical examples of the following definition:

\begin{defn}
Let $K$ be a field. Following Efrat, we say that $K$ is a {\bf $p$-Demushkin field} if $K$ contains $\zeta_p$ and $G_K(p)$ is a Demushkin group. The Demushkin invariants of $G_K(p)$ are in this case labelled $n_K, s_K$ and $\chi_K$.
\end{defn}

\noindent Note that $p$-Demushkin fields are by definition of characteristic different from $p$.

The invariants $n_K, s_K$ and $\chi_K$ of a $p$-Demushkin field $K$ can be naturally interpreted in terms of the field structure, as we now explain.

\begin{defn}
Let $p$ be prime. For a field $F$ and $a \in F^{\times} \setminus (F^{\times})^p$, define
\begin{equation}
N_F(a):=Norm_{F(\sqrt[p]{a})/F}(F(\sqrt[p]{a})^{\times}) \nonumber
\end{equation}

Since $F(\sqrt[p]{a})=F(\sqrt[p]{b})$ whenever $a/b \in (F^{\times})^p$, we will by abuse of notation also write $N_F(a)$ for $a \in F^{\times}/(F^{\times})^p$, with $a$ denoting both an element of $F^{\times}$ and its class modulo $p$-th powers. 
\end{defn}

\begin{rmk}
When the base-field in question is clear, we will just write $N(a)$ for ease of notation.
\end{rmk}

Suppose that $K$ is a $p$-Demushkin field. Since the characteristic of $K$ is not $p$, Kummer Theory can be invoked to provide, upon choosing such a $p$-th root of unity, an isomorphism
\begin{equation}
H^1(G_K(p), \Z / p\Z) \simeq K^{\times}/(K^{\times})^p, \nonumber
\end{equation}
and hence the invariant $n_K$ is precisely the number of independent $p$-th power equivalence classes in $K^{\times}/(K^{\times})^p$. If $K$ is a finite extension of $\mathbb{Q}_p$, then $n = [K:\mathbb{Q}_p]+2$ (\cite{serre2}, Chapter 5, Section 5.6, Theorem 4). A description of the invariant $s_K$ will follow from the following characterisation of $\chi$.
\begin{lemma}
\label{lemma:invariant_chi}
Let $K$ be a $p$-Demushkin field. Then the Demushkin invariant $\chi_K$ coincides with the $p$-cyclotomic character of $K$.
\end{lemma}
\begin{proof}
This is proven in \cite{labute} assuming that $K$ is a finite extension of $\mathbb{Q}_p$ containing $\zeta_p$ (cf. the discussion following Theorem 7 of ibid), but the same proof works for more general $K$. The point is that one can show that the cyclotomic character $\chi'$ of $G_K(p)$ satisfies the property that the induced homomorphisms $H^1(G, I_j(\chi')) \rightarrow H^1(G, I_1(\chi'))$ are surjective for $j \geq 1$, where $I_j(\chi')$ is the $G$-module defined by letting $G$ act on $\Z / p^j \Z$ via $\chi'$. Since $\chi_K$ is the unique homomorphism satisfying this (Theorem 4 of \cite{labute}), it follows that $\chi_K = \chi'$.
\end{proof}

\begin{cor}
\label{cor:invariant_s}
Let $K$ be a $p$-Demushkin field. The invariant $s_K \in \mathbb{N} \cup \{\infty\}$ is uniquely determined in terms of the highest power $p^s$ for which $\zeta_{p^s} \in K$.
\end{cor}
\begin{proof}
As previously stated, $s$ is the highest power of $p$ for which $\text{Im}(\chi) \subset 1 + p^s\Z_p$. But by the above lemma, $\chi$ corresponds to the cyclotomic character, so this condition says that $p^s$ is the maximal power for which the cyclotomic character acts trivially on $\Z_p/p^s\Z_p$. This is exactly equivalent to $p^s$ being the maximal power for which $\zeta_{p^s} \in K$, essentially by definition of the cyclotomic character.
\end{proof}

\noindent 




\begin{rmk}
\label{rmk:diff_gal_groups}
The classification of Demushkin groups now makes it clear why different extensions of $\mathbb{Q}_p$ can have the same pro-$p$ Galois group. For example, if $p \neq 2$, let $F$ and $K$ be two finite degree $n$ extensions of $\mathbb{Q}_p$ containing $\zeta_p$, for some integer $n$. Then case C1 of the classification shows that $G_F(p) \simeq G_K(p)$ as long as $s_K = s_F$, that is, as long as they have the same $p^m$-th roots of unity.
\end{rmk}

\begin{rmk}\label{cool}
Note that if $K=\mathbb{Q}_p(\zeta_p)$, then if $K'$ is an arbitrary finite extension of $\mathbb{Q}_p$ containing $\zeta_p$, we have $G_{K'}(p) \simeq G_K(p)$ if and only if $K=K'$, as $K$ is the only such extension with Demushkin invariants $n_K=p-1, s_K=1$. Therefore in this case, $F'$ may be taken to be $F$ in the statement of Conjecture 1.
\end{rmk}

Next, we want to describe the cup product pairing. The theory of Brauer groups provides an isomorphism
\begin{equation}
H^2(G_K(p), \Z / p\Z) \simeq \Z / p\Z \simeq {_{p}Br(K)} ,\nonumber
\end{equation}
where ${_{p}Br(K)}$ is the $p$-torsion subgroup of the Brauer group of $K$. Here, the first isomorphism comes from the Demushkin assumption and is non-canonical, while the isomorphism of $H^2$ with ${_{p}Br(K)}$ is canonical.

\begin{defn} (`Hilbert symbol')
\label{def:symbol}
For a field $K$ with $\zeta_p \in K$, and two elements $a,b \in K^{\times}$, the `symbol' $(a,b)_K$ denotes the Brauer equivalence class of the central simple $K$-algebra defined by two generators $x,y$ subject to the relations $x^p=a, y^p=b, xy=\zeta_p yx$. If $p=2$ these are exactly the quaternion algebras over $K$.
\end{defn}
A given equivalence class $(a,b)_K$ is trivial iff the algebra splits, which can be written informally as $(a,b)_K=1$, where we have used multiplicative notation for the cyclic group $\Z / p\Z$. We have $(a,b)_K=1$ iff $a \in N(b)$ iff $b \in N(a)$. The symbol $(a,b)_K$ is sometimes informally referred to as the `Hilbert symbol'.

The cup-product pairing can now be identified with the map
\begin{equation}
K^{\times}/(K^{\times})^p \times K^{\times}/(K^{\times})^p \rightarrow \Z / p\Z \nonumber
\end{equation}
sending the pair $a,b$ to the symbol $(a,b)_K$.\\

The above discussion shows that the invariants of a $p$-Demushkin field, and the cup-product, have natural interpretations in terms of the field arithmetic. It turns out that the property of being Demushkin itself, i.e. the criteria on dimensions and the non-degeneracy of the cup-product, can also be interpreted in terms of the field arithmetic. This observation was noted and proved in Frohn's thesis \cite{frohn}, who described it as a sort of `local reciprocity law'. This interpretation will follow as a consequence of the following more general fact from linear algebra

\begin{prop}
\label{linalg}
Let $(\, , ): V \times V \rightarrow W$ be a symmetric or skew-symmetric bilinear pairing, where $V$ and $W$ are finite-dimensional vector spaces over a finite field $L$, such that $W$ is generated by the elements $(v,w)$ for $v,w \in V$. Then the following statements are equivalent:
\begin{itemize}
\item[(1)] The pairing is non-degenerate and $dim_L(W)=1$.
\item[(2)] The map $\Psi: \{U \leq V \, | \, dim_L(U)=1 \} \rightarrow \{U \leq V \, | \, codim_L(U)=1 \}$ given by 
\begin{equation}
\Psi(Lv) = v^{\perp} = \{u \in V \, | \, (v,u)=0 \} \nonumber
\end{equation}
is a bijection, where $Lv$ denotes the 1-dimensional subspace spanned by $v$.
\end{itemize}
\end{prop}
\begin{proof}
For the direction $(1)\Rightarrow (2)$, a non-degenerate bilinear form always induces an isomorphism of $V$ with its dual $V^{*}$, which in turn induces a bijection between 1-dimensional subspaces of $V^{*}$ with the codimension-1 subspaces of $V$. Note that this does not require the field $L$ to be finite.

For the converse direction, where the assumption on $L$ is needed, we refer to \cite{frohn} Theorem 4.24.
\end{proof}

\begin{prop}(`Local Reciprocity for $p$-Demushkin fields')
\label{frohn}
Let $K$ be a $p$-Demushkin field. Then for each $a \in K^{\times} \setminus (K^{\times})^p$, $N(a)$ is a subgroup of $K^{\times}$ of index $p$, and the map
\begin{equation}
\phi: \{K(\sqrt[p]{a}): a \in K^{\times} \setminus (K^{\times})^p\} \rightarrow \{H \leq K^{\times}/(K^{\times})^p : \text{H has index p} \} \nonumber
\end{equation}
given by $K(\sqrt[p]{a}) \mapsto N(a)$ is a bijection between Galois extensions of degree $p$ and subgroups of $K^{\times}/(K^{\times})^p$ of index $p$. Conversely, any field $K$ containing $\zeta_p$ for which $\phi$ is a bijection is necessarily Demushkin.
\end{prop}
\begin{proof}
Apply Proposition \ref{linalg} in the case where $V = H^1(G_K(p), \Z / p\Z), W = H^2(G_K(p), \Z / p\Z)$, both being finite-dimensional $\mathbb{F}_p$-spaces, and the pairing is the cup-product. Here we note that the Merkurjev-Suslin theorem implies that $W \simeq {_{p}Br(K)}$ is indeed generated by the equivalence classes $(a,b)_K$.
\end{proof}

\begin{rmk}
In the case $p=2$, we get from this that in a 2-Demushkin field, $N(a)=N(b)$ if and only if $a$ and $b$ are \emph{equal} modulo squares. This will be crucially exploited in what follows.
\end{rmk}

\subsection{The structure of $G_{\Q}(2)$.}

Theorem 1 concerns the group $G_{\Q}(2)$. This group evidently falls under case C2 of the classification (Theorem \ref{thm:classification}), but the value of $f$ is not immediately apparent. This ambiguity is taken care of by the following result of Serre.

\begin{theorem} (Serre)
Let $K$ be a finite extension of $\Q$ with $N=[K:\Q]$ odd. Then $G_K(2)$ is generated by $N+2$ elements subject to the single relation
\begin{equation}
r = x_1^2 x_2^4 [x_2,x_3]...[x_{N+1},x_{N+2}]. \nonumber
\end{equation}
\end{theorem}
\begin{proof}
This is given as Theorem 8 in \cite{labute}. The key point is that extensions of odd degree do not add any information about primitive $2^N$-th roots of unity. Hence $Im(\chi)$ is as big as possible, implying that $f=2$.
\end{proof}

\begin{rmk}
\label{rmk:diff_gal_groups_p2}
Thus $G_{F}(2) \simeq G_{F'}(2)$ for any two extensions $F, F'$ of $\Q$ with the same odd degree, irrespective of whether $F \simeq F'$ or not.
\end{rmk}

\begin{cor}
\label{cor:gq2}
The group $G_{\Q}(2)$ has rank 3, and can be generated by three elements $x, y, z$ subject to the single relation
\begin{equation}
x^2 y^4 [y,z] = 1. \nonumber
\end{equation}
The abelianisation is given by
\begin{equation}
G^{ab}_{\Q}(2)\simeq \Z / 2\Z \times \Z_2^2. \nonumber
\end{equation}
\end{cor}
\begin{proof}
The first statement follows immediately by setting $N=1$ in the preceding theorem. The statement concerning the abelianisation follows from Lemma \ref{g_ab_lemma} by noting that the $s$-invariant of $\Q$ is 1.
\end{proof}

\section{Step 1 of the main result: recovery of multiplicative structure and norm groups}

For the rest of this section and the next we now fix a field $K$ with
\begin{equation}
G_K(2) \simeq G_{\Q}(2). \nonumber
\end{equation}
\noindent In this section we aim to prove that the structure and arithmetic properties of $K^{\times}/(K^{\times})^2$ are essentially the same as those of $\Q^{\times}/(\Q^{\times})^2$.

We first define some useful notation. For $x,y \in L^{\times}$, $L$ a field, $p$ a prime, we write $x \sim y$ if $x$ and $y$ define the same equivalence class in $L^{\times}/(L^{\times})^p$, that is, if $x/y \in (L^{\times})^p$. If $x_1, \ldots, x_n \in L^{\times}$, we write
\begin{equation}
\langle x_1, \ldots, x_n \rangle \nonumber
\end{equation}
for the subspace of $L^{\times}/(L^{\times})^p$ generated by $x_1(L^{\times})^p, \ldots, x_n(L^{\times})^p$. So as a multiplicative group,
\begin{equation}
\langle x_1, \ldots, x_n \rangle = \{ x_1^{i_1}\cdots x_n^{i_n}(L^{\times})^p \,|\, i_j \in \{0, 1, \ldots, p-1\} \}. \nonumber
\end{equation} 

\begin{prop}
The group $\Q^{\times}/(\Q^{\times})^2$ has dimension 3 when viewed as a vector space over $\mathbb{F}_2$. A basis is given by the square classes of $-1,2$ and $5$. Hence $\Q^{\times}=\{\pm 1, \pm 2, \pm 5, \pm 10\}$ modulo squares.
\end{prop}
\begin{proof}
The dimension of the vector space matches the rank of $G_{\Q}(2)$, which we know to be $3$ by Corollary \ref{cor:gq2}. The fact that $-1, 2$ and $5$ form a basis is explained in e.g. \cite{serre3} (Chapter 24, Section 4: see the discussion following Lemma 3).
\end{proof}

We will show that the same is true for $K^{\times}/(K^{\times})^2$. Indeed, we will show the stronger statement that any relation between square classes of \emph{algebraic} elements in $\Q$ also holds in $K$. For example, in $\Q$, it is true that $3 \sim -5$, so the same will also hold in $K$. This will be made precise in Proposition \ref{normprop} below. In order to do so we will want to make use of the field-theoretic description of the Demushkin structure of $G_K(p)$ described in the previous section. We therefore first need to prove that $K$ is genuinely 2-Demushkin, or in other words that $\text{char}(K) \neq 2$. As a preliminary step towards this and more, we prove the following lemma.

\begin{lemma}\label{lemma1}
\begin{itemize}
\item[(i)] There is a quadratic extension of $\Q$ which does not embed into a $\Z / 4\Z$-extension of $\Q$, that is, a Galois extension $L/\Q$ with Galois group $Gal(L/\Q) \simeq \Z / 4\Z$.
\item[(ii)] Any finite extension $F/\Q$ containing $\zeta_8$ must have degree at least 4.
\end{itemize}
\end{lemma}
\begin{proof}
By Corollary \ref{cor:gq2},
\begin{equation}
G^{ab}_{\Q}(2)\simeq \Z / 2\Z \times \Z_2^2. \nonumber
\end{equation}
By the fundamental Galois correspondence, we need to prove that there exists an index-2 (closed) subgroup $H$ of $G^{ab}_{\Q}(2)$ such that for any subgroup $H' < H$ one has $G^{ab}_{\Q}(2)/H' \not \simeq \Z/4\Z$. Picking $H$ to be the $\mathbb{Z}_2^2$ component of $G^{ab}_{\Q}(2)$ clearly qualifies. It has index 2, and if $H'$ is any  finite index subgroup of $H$ then $G^{ab}_{\Q}(2)/H'\simeq \Z/2\Z \times N$ for some non-trivial 2-group $N$: such a group can never be isomorphic to $\Z/4\Z$.

For (ii), let $F$ be such an extension. Since $\zeta_8 = 1/\sqrt{2} + i/\sqrt{2}$ and $\sqrt{2} \not \in \Q(i)$, $F$ must have degree at least 4.
\end{proof}

\begin{cor}
The characteristic  of $K$ is not 2, and $-1 \not \in K^2$. In particular, $2 \in K^{\times}$.
\end{cor}
\begin{proof}
Since the property $(i)$ in the previous lemma is equivalent to a group-theoretic statement about $G_{\Q}(2)$, it is also true of $K$. But if $\text{char}(K)=2$, or $-1 \in K^2$, then every (separable) quadratic extension \emph{does} embed into a $\Z / 4\Z$-extension. When  $\text{char}(K)=2$ this follows immediately from a theorem of Witt (\cite{witt1936konstruktion}, Satz, page 237). If $\text{char}(K) \neq 2$ and $-1 \in K^2$, then consider an arbitrary quadratic extension $K(\sqrt{a})$ of $K$. Then $Norm_{K(\sqrt{a})/K}(\sqrt{a}) = -a \not \in K^2$, as $-1 \in K^2$ and $a \not \in K^2$. Thus $\sqrt{a} \not \in K(\sqrt{a})^2$, and hence $K(\sqrt{a})$ is contained in the $\Z / 4\Z$-extension $K(\sqrt[4]{a})$.  \end{proof}


We can now freely invoke the description of the Demushkin invariants $n_K$, $s_K$ and $Im(\chi_K)$ from Section \ref{sec:demushkin_fields}. To start with, $n_K = n_{\Q}$ and hence
\begin{equation}
dim_{\mathbb{F}_2} \frac{K^{\times}}{(K^{\times})^2} = dim_{\mathbb{F}_2} \frac{\Q^{\times}}{(\Q^{\times})^2} = 3. \nonumber
\end{equation}

The following Lemma will show that -1 and 2 are independent, non-trivial square classes in $K^{\times}/(K^{\times})^2$, and that the characteristic of $K$ is in fact 0.

\begin{lemma}\label{squareclasses}
\begin{itemize}
\item[(i)] $2 \not \in K(\sqrt{-1})^2$
\item[(ii)] $-1$ is not the sum of two squares, and char($K$)=0.
\end{itemize}
\end{lemma}
\begin{proof}

To prove (i), observe that the statement of Lemma \ref{lemma1}(ii) is purely group theoretic. Indeed, by Corollary \ref{cor:invariant_s}, the statement of Lemma \ref{lemma1}(ii) is equivalent to the statement that any field extension $F/\Q$ with $s_F \geq 3$ has degree $[F:\Q] \geq 4$. This is equivalent to the statement that any closed, normal subgroup $H$ of $G_{\Q}(2)$ with index $<3$ must have an $s$-invariant $<3$. Since the $s$-invariant admits a purely group theoretic characterisation in terms of the isomorphism invariant $\chi$ (see discussion after Lemma \ref{g_ab_lemma}), this statement is purely group theoretic. Since $G_K(2) \simeq G_{\Q}(2)$, it follows that the same statement is true for $K$, which proves the claim.

For $(ii)$, note that by $(i)$, it follows that $-1$ and $2$ are independent and non-trivial square classes in $\sq$. If $-1$ were a sum of two squares, that is, if $-1 \in N(-1)$, then since $2 \in N(-1)$ also, $N(-1) = \langle -1, 2 \rangle$. So $-2 \in N(-1)$ whence $-1 \in N(-2)$, and so $N(-2) = \langle -1, 2 \rangle = N(-1)$. By Proposition \ref{frohn} and Remark 2.14, $2$ is a square: contradiction.

In particular, since in any finite field, $-1$ is a sum of two squares, the characteristic of $K$ must be 0. 
\end{proof}

By the above lemma, we may now define the field $k := K \cap \overline{\mathbb{Q}}$. That is, $k$ is the relative algebraic closure of $\mathbb{Q}$ in $K$. Our next goal is to elucidate the structure of $k$, and show that except for one `bad case', $k$ admits a `2-adic' valuation, i.e., a valuation such that the henselization $k^h$ of $k$ is isomorphic to $\Q \cap \overline{\mathbb{Q}}$. Note that this latter field can be identified with the henselization $\mathbb{Q}^h$ of $\mathbb{Q}$ with respect to the 2-adic valuation, and is elementarily equivalent to $\Q$ (cf. e.g \cite{proq}). In particular, $k$ is a subfield of $\Q \cap \overline{\mathbb{Q}}$.
\\

Before proceeding with the next proposition, let us recall some results about extensions of $p$-adic fields. 

\begin{lemma}
\label{henselizationlemma}
Let $L/\mathbb{Q}_p$ be an algebraic extension containing a primitive $p$-th root of unity $\zeta_p$.
\begin{itemize}
\item[(i)] If the extension is finite, then $dim_{\mathbb{F}_p} L^{\times}/(L^{\times})^p =[L:\mathbb{Q}_p]+2$.
\item[(ii)] If $L/\mathbb{Q}_p$ is a Galois extension 
whose degree $[L:\mathbb{Q}_p]$ is divisible by $p$, then any element in ${_{p}Br(\mathbb{Q}_p)}$ of order $p$ becomes trivial in ${_{p}Br(L)}$.
\end{itemize}
Both statements are also true if we replace $\mathbb{Q}_p$ with $\mathbb{Q}^h$, the henselization of $\mathbb{Q}$ with respect to the $p$-adic valuation.\footnote{This is an example of the well known slogan that as far as algebra is concerned, henselizations are as good as completions.}
\end{lemma}
\begin{proof}
See e.g. \cite{serre2} Chapter 2, Section 5.6, Theorem 4 for the proof of $(i)$. The proof of $(ii)$ follows from the same argument used to show that if $p^{\infty}$ divides $[L:\mathbb{Q}_p]$, then ${_{p}Br(L)}=0$, and can be found in Corollary 7.1.4 of \cite{Neukirch2007} (Chapter 7). Namely, under the identification of the Brauer groups with $\mathbb{Q}/\Z$, the inclusion  $\mathbb{Q}_p \rightarrow L$ corresponds to the homomorphism $Br(\mathbb{Q}_p) \rightarrow Br(L)$ given by multiplication by the degree $[L:\mathbb{Q}_p]$. If $p$ divides this degree, then multiplying by it will kill any order $p$ elements of ${_{p}Br(\mathbb{Q}_p)}$.

To show that we may replace $\mathbb{Q}_p$ by $\mathbb{Q}^h$ in both statements, we recall the well-known consequence of Krasner's Lemma that $\overline{\mathbb{Q}}_p = \overline{\mathbb{Q}}\mathbb{Q}_p$, and so the restriction $G_{\mathbb{Q}_p} \rightarrow G_{\mathbb{Q}^h}$ is an isomorphism. Since $dim_{\mathbb{F}_p} L^{\times}/(L^{\times})^p = dim_{\mathbb{F}_p}H^1(G_L(p), \Z/p\Z)$ and ${_{p}Br(L)} \simeq H^2(G_L(p), \Z/p\Z)$, the claim follows.
\end{proof}

We will also require the following classical result of class field theory, where we are using the symbol notation introduced in Definition \ref{def:symbol}.

\begin{theorem}(Albert-Brauer-Hasse-Noether)
Let $a,b \in k$, with $k$ a number field containing $\zeta_p$. Then the symbol $(a,b)_k=1$ if and only if $(a,b)_{\hat{k}_v}=1$ for every completion $\hat{k}_v$ with respect to a valuation or an ordering $v$ on $k$. That is, the central simple algebra $(a,b)_k$ splits if and only if it splits over every completion $k_v$.
\end{theorem}

\begin{rmk}
As with the above lemma, the statement is still true if we consider henselizations and real closures rather than completions, by similar reasoning.\\
\end{rmk}

We are now ready for the crucial lemma.

\begin{lemma}\label{keylemma}
Given $k=K \cap \overline{\mathbb{Q}}$, one of the following cases holds:
\begin{itemize}
\item[(A)] The restriction map $G_K(2) \rightarrow G_k(2)$ is an isomorphism, and there exists a valuation on $k$ with henselization $k^h = \overline{\mathbb{Q}} \cap \mathbb{Q}_2$;
\item[(B)] $k$ is formally real and $G_k(2)$ is not Demushkin. Furthermore, one of the following is true:
\begin{itemize}
    \item[$B_1$:] $3 \sim 1$ or $2$ in $k$.
    \item[$B_2$:] $3$ is independent of $\pm 1, \pm 2$ in $k$, $N_k(-2) = \langle 2,3 \rangle$, and $k^{\times}/(k^{\times})^2 = \langle -1,2,3 \rangle$.
\end{itemize}
\end{itemize}
\end{lemma}
\begin{proof}
Choose any chain of number fields $k_0=\mathbb{Q} \subseteq k_1 \subseteq \ldots \subset k$ such that $k = \bigcup_{i=0}^{\infty} k_i$. By Lemma \ref{squareclasses}, $-1 \not \in N_K(-1)$, so also $-1 \not\in N_k(-1)$ and $-1 \not\in N_{k_i}(-1)$ for any $i$.  If we let $\Sigma_i$ denote the set of all orderings and valuations $v$ of $k_i$ for which $-1 \not \in N_{k_{i}^v}(-1)$, where $k_{i}^v$ is a real-closure, resp. a henselization, of $k_i$ with respect to $v$, then by the Albert-Brauer-Hasse-Noether Theorem every $\Sigma_i$ is non-empty. For $i<j$, each valuation in $\Sigma_j$ lies above a valuation in $\Sigma_i$. Now, it is easy to see that $\Sigma_0$ contains the ordering of $\mathbb{Q}$. Since it is only for $p=2$ that the Hilbert symbol $(-1,-1)_p =-1$, we see that $\Sigma_0$ consists of exactly the ordering and the 2-adic valuation, and hence every valuation in $\Sigma_i$ lies above one of these. Each $\Sigma_i$ is also finite, due to the well-known fact that the Hilbert symbol trivialises in the Henselization of all but finitely many valuations. Since the $\Sigma_i$ form an inverse system of finite, non-empty sets, their inverse limit $\Sigma_{\infty}$ is non-empty. Further, every valuation $v \in \Sigma_{\infty}$ is either archimedean (corresponding to an ordering) or dyadic, as it must lie over one of these two valuations on $k_0$. We now distinguish between two cases.

{\bf Case A:} Suppose that $\Sigma_{\infty}$ contains a dyadic valuation $v$. If we let $k^h$ denote the henselization of $k$ with respect to $v$, then $-1 \not \in N_{k^h}(-1)$. If we denote by $\mathbb{Q}^h$ a henselization of $\mathbb{Q}$ with respect to the 2-adic valuation (which we may without loss of generality take to be $\mathbb{Q}_2 \cap \overline{\mathbb{Q}}$) then there is a natural embedding $\mathbb{Q}^h \hookrightarrow k^h$. For notational convenience, we set $F:=k^h$ and $L=\mathbb{Q}^h$. We claim $F=L$.

Indeed, first notice that if $[F:L]$ is even, then all order-2 elements of ${_{2}Br(L)}$ become trivial in $F$, by Lemma \ref{henselizationlemma}(ii). But the Brauer class of $(-1,-1)$ is of order 2 in both $L$ and $F$, so this cannot be the case. Therefore, $F/L$ has odd, possibly infinite, degree. The odd degree implies that the canonical map
\begin{equation}
L^{\times}/(L^{\times})^2 \rightarrow F^{\times}/ (F^{\times})^2
\end{equation}
is injective, and the same is true if we replace $L$ by any finite subextension $L'/L$ of $F$. Hence for any such $L'$,
\begin{equation}
dim_{\mathbb{F}_2} \frac{L'^{\times}}{(L'^{\times})^2} \leq dim_{\mathbb{F}_2} \frac{F^{\times}}{(F^{\times})^2}. \nonumber
\end{equation} 

We now claim that the canonical map
\begin{equation}
k^{\times}/(k^{\times})^2 \rightarrow F^{\times}/(F^{\times})^2 \nonumber
\end{equation}
is \emph{surjective}. Indeed, if $\hat{k}$ and $\hat{F}$ are the completions of $k$ and $F$, then we have a commutative diagram
\begin{center}
\begin{tikzpicture}
  \matrix (m) [matrix of math nodes,row sep=3em,column sep=4em,minimum width=2em]
  {
     k^{\times}/(k^{\times})^2 & F^{\times}/(F^{\times})^2 \\
     \hat{k}^{\times}/(\hat{k}^{\times})^2 & \hat{F}^{\times}/(\hat{F}^{\times})^2 \\};
  \path[-stealth]
    (m-1-1) edge node [left] {} (m-2-1)
            edge [] node [below] {} (m-1-2)
    (m-2-1.east|-m-2-2) edge node [below] {}
            node [above] {$=$} (m-2-2)
    (m-1-2) edge node [right] {} (m-2-2);
\end{tikzpicture}
\end{center}
where the arrow at the bottom is an equality since $\hat{k}$ is also the completion of $k^h$. Since $k$ is dense in its completion and the non-zero squares form an open subset of $\hat{k}$, the vertical arrow on the left is a surjection. Furthermore, the vertical arrow on the right is injective: if $a \in F^{\times}$ is a square in $\hat{F}^{\times}$, then the polynomial $x^2-a$ satisfies the conditions of Hensel's lemma and is hence already solvable over the henselisation $F$. The commutativity of the diagram then implies that the horizontal arrow on the top must be a surjection, as claimed.

Now, note that since $k \subset K$,
\begin{equation}
dim_{\mathbb{F}_2} \frac{k^{\times}}{(k^{\times})^2} \leq dim_{\mathbb{F}_2} \frac{K^{\times}}{(K^{\times})^2} = 3. \nonumber
\end{equation} 
Combining this with the surjection just established, we get
\begin{equation}
dim_{\mathbb{F}_2} \frac{L'^{\times}}{(L'^{\times})^2} \leq dim_{\mathbb{F}_2} \frac{F^{\times}}{(F^{\times})^2} \leq dim_{\mathbb{F}_2} \frac{k^{\times}}{(k^{\times})^2} \leq 3 \nonumber
\end{equation}
for any finite subextension $L'/L$ of $F$. But by Lemma \ref{henselizationlemma}, 
\begin{equation}
dim_{\mathbb{F}_2} L'^{\times}(L'^{\times})^2 = [L':\mathbb{Q}^h] + 2 \geq 3
\end{equation}
and so it must be that $[L':\mathbb{Q}^h]=1$. As this is true for any such finite subextension $L'/L$ of $F$, it follows that $F = \mathbb{Q}^h$. Thus $k \hookrightarrow k^h = \mathbb{Q}^h = \Q \cap \overline{\mathbb{Q}}$. Additionally, we see that \begin{equation}
dim_{\mathbb{F}_2} \frac{k^{\times}}{(k^{\times})^2} = 3. \nonumber
\end{equation}

From this equality and the fact that $-1, 2$ and $5$ form a basis for ${\mathbb{Q}^h}^{\times}/({\mathbb{Q}^h}^{\times})^2$, it follows that $-1, 2$ and $5$ also form a basis for $k^{\times}/(k^{\times})^2$. We also know that $3 \sim -5$ in $k$, since this is true in $k^h$. Since $3 = 1+2$ is visibly in $N_k(-2)$, we have $-5 \in N_k(-2)$. Since $G_{k^h}(2)$ is Demushkin of rank 3, the norm groups of $k$ are generated by exactly 2 elements, by local reciprocity (Proposition \ref{frohn}). We now have enough information to work out all the norm groups $N_k(a)$ for $a \in \{-1,\pm 2, \pm 5, \pm 10\}$. They are as follows:
\begin{itemize}
\item $N_k(-1) = \langle 2,5 \rangle$
\item $N_k(2) = \langle -1,2 \rangle$
\item $N_k(5) = \langle -1,5 \rangle$
\item $N_k(10) =\langle -1, 10 \rangle$
\item $N_k(-2)=\langle 2,-5 \rangle$
\item $N_k(-5)=\langle -2,5 \rangle$
\item $N_k(-10)=\langle -2,-5 \rangle$
\end{itemize}
Indeed, it is easy to check that the right-hand sides are all contained in the left-hand sides. The only generators not clearly a norm of the right shape are $-5 \in N_k(-2)$ and equivalently $-2 \in N_k(-5)$, which we considered above. The rest are routine (e.g. $2=1^2+1^2, 5 = 1^2+2^2$, so $\langle 2,5 \rangle \subset N_k(-1)$). The generators are independent because they are independent already in the ambient space ${\mathbb{Q}^h}^{\times}/({\mathbb{Q}^h}^{\times})^2$. Finally, none of the norm-groups  on the left-hand side can be \emph{strictly} bigger than the corresponding right-hand side, as that would imply they are bigger also in $k^h$, which we know they are not.

It follows by Proposition \ref{frohn} that $k$ is Demushkin of rank 3. Since $k$ is relatively algebraically closed in $K$, and the primitive $2^N$-th roots of unity ($N \geq 1$) are already algebraic over $\mathbb{Q}$, we have $Im(\chi_K) = Im(\chi_k)$. Hence $G_K(2)$ and $G_k(2)$ are Demushkin groups with the same invariants, and so are isomorphic finitely generated pro-2 groups. Thus the epimorphism $G_K(2) \rightarrow G_k(2)$ is an isomorphism, by the profinite pigeon-hole principle (see \cite{ribeszal} Proposition 2.5.2).

{\bf Case B:} Suppose $\Sigma_{\infty}$ does not contain a dyadic valuation. Then $k$ is formally real. If $dim_{\mathbb{F}_2} k^{\times}/(k^{\times})^2 = 2$, then  $k^{\times}/(k^{\times})^2 = \langle -1,2 \rangle$. Because sums of squares are always positive, $N_k(-1) = \langle 2 \rangle$. Also, $1+2a^2 > 0$ for any $a \in k$, hence $N_k(-2) = \langle 2 \rangle$. But we always have $-1$ and $2$ in $N_k(2)$, due to the trivial identity $-1 = 1-2$. Therefore $N_k(2) = \langle -1, 2  \rangle$. Since the norm-groups have different size, it follows from Theorem \ref{frohn} that $k$ is not Demushkin. This puts us in sub-case $B_1$, where we note that the identity $3 = 1+2$ means $3 \in N_k(-2) = \langle 2 \rangle$, i.e. $3 \sim 1$ or $3 \sim 2$.

Suppose instead that $dim_{\mathbb{F}_2} k^{\times}/(k^{\times})^2 = 3$. Then $k^{\times}/(k^{\times})^2 = \langle -1,2,c \rangle$ for some $c \in k^{\times}$. If $N_k(-1)$ is a subspace of dimension 2 then we may assume, replacing $c$ by $-c$ if necessary, that $N_k(-1) = \langle 2,c \rangle$. In this case, by reciprocity, $-1 \in N_k(c)$. Since also $c \in N_k(c)$, and the subspace $\langle -1,c \rangle$ of $N_k(c)$ has dimension 2, the equality $N_k(c) = \langle -1,c \rangle$ is forced. It follows that $-2 \not \in N_k(c)$, and hence, reciprocally, $c \not \in N_k(-2)$. Also, since both $N_k(-1)$ and $N_k(-2)$ contain elements that are positive with respect to any ordering, we deduce firstly that $c \in N_k(-1)$ is positive and secondly that $-c \not \in N_k(-2)$. We can therefore conclude that $N_k(-2) = \langle 2 \rangle$. In particular, $k$ is not Demushkin, as before, and we are again in sub-case $B_1$.

If, on the other hand, $N_k(-1)$ is a subspace of dimension 1, then necessarily $N_k(-1) = \langle 2 \rangle$. In this case $k$ is once again not Demushkin, as $N_k(2) = \langle -1,2 \rangle$ is always true, so the norm groups have differing dimensions. Since an equality $N_k(-2) = \langle 2 \rangle$ would put us back in sub-case $B_1$, let's instead assume that $3$ is independent of $1$ and $2$, and hence $N_k(-2) = \langle 2,3 \rangle$. Because $3$ is necessarily positive with respect to the ordering, $3$ cannot be equivalent to $-1$ or $-2$ modulo squares. We may therefore, without loss of generality, set $c=3$ in this case, to get $k^{\times}/(k^{\times})^2 = \langle -1,2,3 \rangle$. This is precisely sub-case $B_2$.
\end{proof}

The appearance of the exceptional Case B essentially emerges from the ambiguity around the square class of $3$ in $K$. Since several of the dihedral extensions of degree 8 of $\Q$ are generated using $\sqrt{3}$ (see \cite{naito}), it may be possible to determine this square class directly from the structure of $G_K(2)$, as was done for $-1$ and $2$. While we were unable to do this, we will never the less prove that

\begin{prop}
\label{caseb}
Case B in Lemma \ref{keylemma} cannot occur.
\end{prop}

\noindent The proof may be found in Appendix B, and is obtained by showing that $k$ being formally real is in all cases incompatible with $K$ being Demushkin. Because the proof proceeds by using the `norm-combinatorics' machinery developed in the next section, the reader is recommended to leave the verification of this Proposition to the end. One may reasonably ask whether a group theoretic proof can be given of Proposition \ref{caseb}, for example by proving that $G_k(2)$ has to be Demushkin. It is possible to identify several potential structures of $G_k(2)$ by (i) using the arithmetic information (existence of orderings/valuations) to identify the potential Witt rings of $k$ and (ii) referring to Tables 5.2 and 5.3 of \cite{jacobware}, which give the isomorphism type of $G_k(2)$ for each such Witt ring. However, in all such cases the authors considered, the abstract group $G_k(2)$ obtained does occur as a quotient of $G_K(2)$, suggesting that Case B cannot be straightforwardly ruled out by purely group-theoretic arguments. It seems that some form of direct appeal to the field structure is necessary.

We now proceed with Case A.

\begin{prop}
\label{normprop}
An $\F$-basis for $\sq$ is given by the classes of $-1, 2$ and $5$. For any $q \in k$, $q \sim 1$ in $K$ if and only if $q \sim 1$ in $\Q$ (in particular, this holds for all $q \in \mathbb{Q}$). The quadratic norm groups $N(a)$ of $K$ are as follows:
\begin{eqnarray}
N_K(-1) &=& \langle 2,5 \rangle \nonumber \\
N_K(2) &=& \langle -1,2 \rangle  \nonumber \\
N_K(5) &=& \langle -1,5 \rangle  \nonumber \\
N_K(10) &=& \langle -1, 10 \rangle  \nonumber \\
N_K(-2) &=& \langle 2,-5 \rangle  \nonumber \\
N_K(-5) &=& \langle -2,5 \rangle  \nonumber \\
N_K(-10) &=& \langle -2,-5 \rangle  \nonumber
\end{eqnarray}
\end{prop}
\begin{proof}
By Proposition \ref{caseb}, we are necessarily in Case A of Lemma \ref{keylemma}. In that case, $-1,2$ and $5$ form a basis for $k^{\times}/(k^{\times})^2$ and they are therefore independent modulo squares. Since $\sq$ has dimension 3, these also form a basis of $\sq$. The structure of the norm groups for $K$ must be the same as that of $k$, since $\sq$ has the same basis as $k$ and the norm groups have the same size. These were calculated in the proof of the above lemma, resulting in the above list. For the last part, note that $q \sim 1$ in $K$ iff $q \sim 1$ in $k$, since if $q$ were a non-square in $k$, it could only become a square in $K$ if one of  $\pm1, \pm2, \pm5, \pm10$ become square in $k$, which we know can't happen. Since $k$ embeds into $\Q$, the same argument shows that $q \sim 1$ in $k$ iff $q \sim 1$ in $\Q$.
\end{proof}

\section{Step 2 of the main result: construction of the valuation via norm combinatorics}

In $\Q$, we can detect the valuation ring via norms by the equality
\begin{equation}
Norm(\Q(\sqrt{5})^{\times})=\mathbb{Z}_2^{\times}\cdot(\Q^{\times})^2 \nonumber
\end{equation}
which follows from the fact that $\Q(\sqrt{5})$ is the (unique) unramified quadratic extension of $\Q$.
We will use this observation along with the following construction from the theory of `rigid elements' (see e.g. \cite{ep}, Section 2.2.3) to construct the valuation of Theorem 1. The general setup is as follows.

Let $p$ be a rational prime, $F$ a field, $T 
\leqslant F^{\times}$ a subgroup containing $(F^{\times})^p$. Define the sets
\begin{eqnarray}
\val_1(T)& :=& \{ x \in F \setminus T : 1+x \in T \} \\
\val_2(T)& : =& \{x \in T : x \val_1(T) \subseteq \val_1(T) \}
\end{eqnarray} and
\begin{equation}
\val(T) := \val_1(T) \cup \val_2(T). \nonumber
\end{equation}

The following key construction of valuations has its roots in the work of Arason, Elman and Jacob (see \cite{arasonelmanjacob}) who first recognized the importance of the condition $1-\val_1(T)\val_1(T) \subseteq T$ (called ``pre-additivity" in their paper) and its connection to valuations.

\begin{lemma}\label{rigid} Suppose that the subgroup $T$ satisfies the following conditions.\\
\begin{itemize}
\item[(i)] (\emph{Existence of rigid elements}) For any $x \not \in T$, $T+xT \subseteq T \cup xT$
\item[(ii)] For any $x,y \in \val_1(T)$, one has $1-xy \in T$.
\item[(iii)] If $p=2$ then $-1 \in T$.\\
\end{itemize}

Then $\val$ is a valuation ring of $F$ with $\val^{\times} \subset T$.
\end{lemma}
\begin{proof}
This follows straightforwardly from Theorem 2.2.7 in \cite{ep} and its proof.
\end{proof}

Consider the above valuation construction with $p=2$, $F=K$ and $T = N(5)$. In this case we write $\val_1$ instead of $\val_1(T)$ etc. Notice that for $K = \Q$, the non-zero elements in $\val_1$ are those with positive, odd valuation, by the ultrametric inequality, and so $\val_2$ consists of the $2$-adic integers with even valuation. Therefore in this case $\val$ is precisely $\mathbb{Z}_2$, showing that this construction does reproduce the 2-adic valuation on $\Q$. We will show that the condition of the lemma holds for our abstract $K$ as well, and then deduce that the valuation ring $\val(N(5))$ satisfies the additional properties desired.

Condition (iii) of Lemma \ref{rigid} is trivially satisfied since $-1 \in N(5)$. Condition (i) is also straightforward when $p=2$. Indeed, for $x \not \in N(5)$, $N(5)$ and $xN(5)$ are distinct cosets of size 4, and hence their union has size 8 and thus spans the whole space: condition (i) is therefore trivially satisfied. The difficult condition is (ii), which is equivalent to the `pre-additivity' condition as a result of condition (i) (Lemma 2.6 of \cite{arasonelmanjacob}).

\begin{rmk}
The statement in \cite{ep} Theorem 2.2.7 for $p=2$ does not include condition (ii) on expressions of the form $1-xy$. Instead, it is shown that, when $p=2$, there necessarily exists a subgroup $T_1$ of $K^{\times}$ satisfying condition (ii), and that for this $T_1$ one has that $\val(T_1)$ is a valuation ring. However, there appears to be no obvious way to exclude the possibility that the subgroup $T_1$ is in fact the whole of $K^{\times}$. That is, there is no way of telling if the valuation obtained is trivial or not. To show that $T_1$ may be taken to be $T$ (i.e., to show non-triviality), one must show that condition (ii) in fact holds for $T$ itself.\\
\end{rmk}

\begin{rmk}
On the other hand, when $p>2$, the proof in \cite{ep} Theorem 2.2.7 shows that condition (i) implies condition (ii). Verifying the existence of rigid elements is therefore likely to be the key challenge in attempts to extend Theorem 1 to primes $p>2$. \\
\end{rmk}

In order to show that $N(5)$ satisfies condition (ii) of Lemma \ref{rigid}, the idea is to decompose the term $1-xy$ in several ways, all of which are visibly in certain norm groups $N(a)$. Working on a case by case basis, depending on the square classes of $x, y, 1+x$ and $1+y$, this places $1-xy$ in the intersection of several norm groups, which are known by Proposition \ref{normprop}. In all cases, the possible square classes of $1-xy$ thus obtained are always in $N(5)$. As an intermediate step, we need to establish that $\pm1, \pm5$ and $\pm 1/5 \in \val_2$, i.e., that these numbers are `units' in $\val$; of course we expect this to be true since these numbers are units in $\Z_2$. Doing this amounts to computing the square class of expressions $1+ax$ when $x \in \val_1$ and $a \in \{\pm1, \pm5, \pm 1/5 \}$. This is again done by writing $1+ax$ as a norm in several different ways, thereby severely restricting its possible square class. Chief among the identities used repeatedly is the \emph{Steinberg  relation}:
\begin{equation}
1 - x \in N(x) \nonumber
\end{equation}
for any $x \not=0,1$. We will use this identity without comment in all calculations.

These calculations show that the square class of expressions like $1+ax$, for $a \in k$, $x \in K$, is determined entirely by the square class of $x, 1+x$ and the `lattice' of norm-groups. If such a statement could be made rigorous and then proved, one could deduce that $\val(N(5))$ is a valuation ring simply because it is one for $k$. Unfortunately, such a structural proof still eludes us, and we instead resort to direct computations.

The full computations and details can be found in Appendix A.
\\

\begin{rmk}
Notice that $0 \in \val_1(T)$ for any $T$. In the calculations and lemmas established in the following, we always ignore this case, as it can easily be seen that $0$ will satisfy all the claims made, or that the resulting computation gives $0$, which we know to be in $\val_1$.\\
\end{rmk}

Before we begin, let us for ease of exposition introduce some notation. For $a_i \in K^{\times}$, we write
\begin{equation}
\{a_1, a_2, \ldots, a_n\} \nonumber
\end{equation}
as shorthand for the subset
\begin{equation}
a_1(K^{\times})^2 \cup a_2(K^{\times})^2 \cup \ldots \cup a_n(K^{\times})^2 \nonumber
\end{equation} of $K^{\times}$.
\\
\\
The key lemma, where the bulk of the computations take place, is the following:

\begin{lemma}
\label{lemma24}
Let $x \in \val_1$. Then $1+ax \in N(5)$ for $a = \pm 1, \pm2, \pm4, \pm5, \pm 1/5$. Furthermore, given a specific choice of $x \in \val_1$, the values of these expressions (modulo squares) can be uniquely constrained, and depend only on the value of $x$ and $1+x$ (modulo squares).
\end{lemma}

\begin{proof}
See Appendix A. The proof uses explicit calculations on a case by case basis, as explained above. These are tedious, but elementary, and can be carried out following a straightforward algorithm. Note that the cases $a = \pm 1, \pm 5, \pm 1/5$ are essentially capturing the fact that these numbers are units with respect to the dyadic valuation, while the cases $a = \pm 2, \pm 4$ are capturing the fact that $v(2) > 0$.
\end{proof}

We are now ready to prove the critical 

\begin{prop}
\label{propxy}
For any $x,y \in \val_1$, $1-xy \in N(5)$.
\end{prop}
\begin{proof}
The key point is to note that for any $a \in K$, we have the following decompositions, already present in \cite{arasonelmanjacob}:
\begin{eqnarray}
1-xy &=& (1+ay)\left(1+(1+x/a)\frac{-ay}{1+ay}\right) \nonumber \\
&=& (1+ax)\left(1+(1+y/a)\frac{-ax}{1+ax}\right) \nonumber
\end{eqnarray}
Now define the following subsets of $K^{\times}/(K^{\times})^2$:
\begin{eqnarray}
D^1_a(x,y) &=& (1+ay)N(a(1+x/a)y(1+ay)) \nonumber \\
D^2_a(x,y) &=& (1+ax)N(a(1+y/a)x(1+ax)) \nonumber \\
D_a(x,y) &=& D^1_a(x,y) \cap D^2_a(x,y) \cap N(xy) \nonumber
\end{eqnarray}
Then it is a consequence of the above decompositions that $1-xy \in D_a(x,y)$. By Lemma \ref{lemma24}, for $a =  \pm 1, \pm 5, \pm 1/5$, we can pin down the values of $1+ax, 1+ay, 1+x/a$ and $1+y/a$, and this allows one to compute the sets $D_a(x,y)$.  It turns out that intersecting all the constraints obtained in this way always uniquely pins down the value of $1-xy$ modulo squares, and one can check, case by case, that this value is always in $N(5)$. The details are found in Appendix A.
\end{proof}

\begin{cor}
The set $\val$ is a non-trivial valuation ring of $K$ with residue characteristic $2$.
\end{cor}
\begin{proof}
We have shown that condition (ii) of Lemma \ref{rigid} is satisfied for $T=N(5)$. Condition (iii) is trivially satisfied since $-1 \in N(5)$. Condition (i) is also straightforward: given any $x \not \in N(5)$, one easily verifies that $xN(5)$ = $K^{\times} \setminus N(5)$, and hence $N(5) \cup xN(5) = K^{\times}$. Thus Lemma \ref{rigid} guarantees that $\val = \val(N(5))$ is a valuation ring with $\val^{\times} \subset N(5)$. In particular, the valuation is non-trivial.

Since $2 \not \in N(5)$, the value of $2$ is strictly positive, whence $2$ becomes trivial in the residue field, and hence the residue field has characteristic 2.
\end{proof}
\begin{rmk}
We will choose a valuation $v$ with $\val_v = \val$, and denote the value group, maximal ideal and residue field of $v$ as $\Gamma$, $\M$ and $Kv$ respectively. \\
\end{rmk}

We now elucidate the structure of $\val_2$ further. Put
\begin{equation}
A:=\{x \in N(5) : 1+2x \in N(5)\}
\end{equation}
In $\Q$, this set coincides, by the ultrametric inequality, with the $2$-adic integers with even valuation. Hence we expect the following

\begin{lemma}
$\val_2 = A$.
\end{lemma}
\begin{proof}
The inclusion $\val_2 \subseteq A$ is trivial: if $x \in \val_2$ then as $2 \in \val_1$, we have $2x \in \val_1$ so that $1+2x \in N(5)$.

For the reverse inclusion, note that if $x \in A$ then $2x \in \val_1$. Since $\val_1$ is invariant under multiplication by $a = \pm 1, \pm5$ and $\pm 5^{-1}$, we have that $1+2ax \in N(5)$. Therefore given $x \in A$, $ax$ also satisfies the conditions required to be in $A$, for any $a = \pm 1, \pm5$ and $\pm 5^{-1}$, and so $A$ is invariant under multiplication by these numbers. Suppose therefore that we can prove that for any $x \in A$ with $x \sim 1$, that $x \in \val_2$. Then as $\val_2$ is also invariant under multiplication by those numbers, it easily follows that $x \in \val_2$ also when $x \sim -1$ or $\pm 5$. Because $\val_2 \subset N(5)=\{\pm1, \pm5\}$, these are the only possible square classes of $x$. If $x \sim 1$, then in addition we know $1+2x \in N(5) \cap N(-2) = \{1, -5\}$. Thus we need only consider the following two cases:\\

\noindent \boxed{\textbf{Case 1: } 1+2x \sim 1}\\ 
We need to show that for any $y \in \val_1$, $xy \in \val_1$. Since $\val_1$ is invariant under $\pm 1, \pm 5, \pm 5^{-1}$, as before, we can assume that $y \sim 2$, whence $1-y \in N(5) \cap N(2) =\{\pm1 \}$.

\indent Suppose first $1-y \sim 1$. Then if $1+2x=a^2$, we get $1+2xy=(1-y)+a^2y \in N(-1) \cap N(-2)$. But it's also in $N(5)$ since $-2x, y \in \val_1$ whence $1-(-2x)y \in N(5)$ by Proposition \ref{propxy}. So $1+2xy \sim 1$. Hence $1+5xy = (1+2xy)+3xy \in N(-10) \cap N(10) = \{1,10\}$, using that $3 \sim -5$. If $1+5xy \sim 10$, we get $1+4xy = (1+5xy)-xy \in 2N(5)$. But since $\val$ is a ring, $2xy \in \val_2$ and so $4xy \in \val_1$, whence $1+4xy \in N(5)$, contradiction. So $1+5xy \in N(5)$, whence it follows that $5xy \in \val_1$ and hence so is $xy$.

Suppose now that $1-y \sim -1$. Then $1+2xy \in N(2) \cap N(5) \cap N(-1) = \{1\}$. Arguing as above, $1+5xy \in N(5)$  and we're done also in this case.\\

\noindent \boxed{\textbf{Case 2: }1+2x \sim -5} \\
First suppose $y \sim 2, 1-y \sim 1$. Then $1+2x = -5a^2$ for some $a$, whence $1+2xy=(1-y)-5a^2y$. Since $2x \in \val_1, 1+2xy \in N(5)$ as well. Consequently $1+2xy \in N(5) \cap N(-1) \cap N(10) = \{1\}$.

It follows that $1+xy =(1+2xy)-xy \in N(-2) \cap N(2)=\{1,2\}$. If $1+xy \sim 2$, then $1+4xy = (1+xy)+3xy \in 2N(5)$. But as $4xy \in \val_1$, $1+4xy \in N(5)$ as well, giving a contradiction. Hence $1+xy \sim 1$ and we're done.

Next suppose $y \sim 2, 1-y \sim -1$. Then as above, we get $1+2xy = (1-y)-5a^2y$ for some $a$, and so $1+2xy \in N(5) \cap -N(-10) \cap N(-1) = \{5\}$. Now $1+5xy = (1+2xy)+3xy \in 5N(2) \cap N(-10)=\{-5,10\}$. If $1+5xy \sim 10$, then $1+4xy =(1+5xy)-xy \in -2N(5) \cap N(5)$ which is empty. Hence we must have $1+5xy \sim -5$. That is, $5xy \in \val_1$, so $xy \in \val_1$ as well.

\end{proof}

\noindent Since $\val^{\times} \subset \val_2$, we have

\begin{cor}The units are
\label{corunits}
\begin{equation}
\val^{\times} = \{x \in N(5): 1+2x \in N(5) \text{ and } 2+x \in N(5)\}
\end{equation}
and the maximal ideal is the disjoint union $\mathcal{M}=\val_1 \sqcup B$ where
\begin{equation}
B:= \{x \in N(5): 1+2x \in N(5) \text{ and } 2+x \not \in N(5)\}
\end{equation}
\end{cor}

\begin{prop}
$B=2\val_1$.
\end{prop}
\begin{proof}
If $x \in B$, then as $2+x \not \in N(5)$, $1+x/2 \in N(5)$, so $x/2 \in \val_1$, whence $x \in 2\val_1$. The other direction follows easily from the previous Lemma in a similar fashion.
\end{proof}

\begin{prop}
$v(2)$ is a minimal positive element in $\Gamma$.
\end{prop}
\begin{proof}
Since $\val_1 \subset \M$, $v(2) > 0$. Now suppose we have $x \in \M$ with $v(x) < v(2)$, i.e., $2/x \in \M = \val_1 \cup 2\val_1$. If $2/x \in 2\val_1$, then $x^{-1} \in \val_1$, contradicting $v(x) > 0$. Thus $2/x \in \val_1$. If it were now the case that $x \in \val_1$, then $x \not \in N(5)$ by definition, so that $x = \pm 2, \pm 10$ modulo squares. Dividing by $x^2$ implies that the same is true for $x^{-1}$. Thus $2/x = \pm 1, \pm 5$ modulo squares, and hence $2/x \in N(5)$, contradicting the fact that $2/x \in \val_1$. Hence it must be that $x \in 2\val_1$, say $x = 2y, y \in \val_1$. But then $v(2) > v(x) = v(2y) > v(2)$, which is absurd.
\end{proof}

Since $v(2) > 0$, it is clear that $Kv$ has characteristic 2. We now prove that the residue field is exactly $\mathbb{F}_2$ and that the valuation is 2-henselian.

\begin{prop}\label{residue}
The residue field $Kv = \mathbb{F}_2$.
\end{prop}
\begin{proof}
It is possible to prove, using similar computations as before, that if $x \in \val^{\times}$ then $1+x \in \M$, implying that $\val / \M$ contains only two elements. However, an anonymous reviewer suggested a simpler argument inspired by the Galois characterisation of $p$-adic fields, which we gratefully reproduce here.

First note that if $a, b \in \val$ are such that $1+2a \sim 1+2b$, then $a-b \in \mathcal{M}$. Indeed, suppose $1+2a = \lambda x^2$ and $1+2b = \lambda y^2$ for some $\lambda \in \{\pm 1, \pm 2, \pm 5, \pm 10\}$. Then clearly $v(1+2a) = v(1+2b) = v(\lambda) = 0$. Since $2(a-b) = \lambda(x^2-y^2)$, we have $v(x^2-y^2)>0$. It follows that $x = \pm y + 2z$ for some $z \in \val$, and hence $v(x^2-y^2) = v(\pm 4zy + 4z^2) \geq v(4)$, giving $v(a-b)>0$ as desired.

Now let $a_1, \ldots a_n$ be representatives in $\val$ of the non-zero elements of $Kv$. The above observation shows that the set $\{1+2a_1, \ldots 1+2a_n, 1, 5 \}$ cover $n+2$ distinct square classes, where we have used that $1 = 1+2\cdot 0$, $5 = 1+2\cdot 2$ and that $v(2)$ is minimal. Condition (4.4) implies that every element of this set belongs to $N(5)$, so by multiplying them all by 2, we conclude that there are at least $2n+4$ square classes in total. Since the total number of square classes in $K$ is precisely 8, this implies that $n \leq 2$. $Kv$ cannot therefore be larger than $\mathbb{F}_2$, and so they must coincide.
\end{proof}

\begin{prop}
The valuation $\val$ is $2$-henselian.
\end{prop}
\begin{proof}
It suffices to show that $\val$ extends uniquely to every quadratic extension of $K$, by Fact 2.2. This will follow from the fundamental inequality for valuations (\cite{ep} Theorem 3.3.4) if we can show that $ef=2$ for any quadratic extension $L$ of $K$, with $e$ and $f$ the ramification and inertia degrees of $L/K$. This can be done in a number of ways, e.g. by `lifting' the ramification/inertia degrees from $k$. The following particularly simple approach was suggested by an anonymous reviewer.

Note that for any quadratic extension $L/K$, either $ef=2$ or $ef=1$. So it suffices to show that $ef \neq 1$, i.e. that the extension is not \emph{immediate}. But $L$ is necessarily obtained by adjoining a square root of $\pm 1, \pm 2, \pm 5, \pm 10$, and hence contains as a subfield $\mathbb{Q}(\sqrt{\alpha})$ for $\alpha \in \{\pm 1, \pm 2, \pm 5, \pm 10\}$. However, the prime ideal $2\mathbb{Z}$ of $\mathbb{Z}$ is either inert or ramifies in any such $\mathbb{Q}(\sqrt{\alpha})$, and so $\mathbb{Q}$ along with the 2-adic valuation admit no immediate extensions. Hence neither can $K$.
\end{proof}

We can now complete the proof of the main theorem.

\begin{theorem}
\label{mainthm}
Suppose $G_K(2) \simeq G_{\Q}(2)$. Then there is a valuation $v$ on $K$ which is $2$-henselian, has discrete value group with $v(2)$ as a minimal positive element, residue field $\mathbb{F}_2$ and $[\Gamma:2\Gamma]=2$.
\end{theorem}
\begin{proof}
The only thing remaining to prove is that $[\Gamma:2\Gamma]=2$. Note that $\Gamma = K^{\times}/\val^{\times}$, and $\val^{\times}(K^{\times})^2=N(5)$. Since $v(\val^{\times}(K^{\times})^2)=2\Gamma$, and $N(5)$ has index 2 in $K^{\times}$, this implies that
\begin{equation}
\Gamma / 2\Gamma \simeq \frac{K^{\times}}{\val^{\times}(K^{\times})^2 } \nonumber
\end{equation}
has order 2.
\end{proof}

\begin{rmk}
Upon reviewing the proof, one finds that we didn't need all of $G_K(2)$, only the significantly smaller quotient $Gal(K''/K)$, where $K''$ is the so-called maximal $\Z/2\Z$ elementary meta-abelian extension of $K$. Indeed, Lemma \ref{lemma1} and \ref{squareclasses} clearly only need this quotient, and Pop has shown (Section 2 of \cite{popmeta}, in particular Lemmas 1, 2 and 6) that the Galois cohomology used is already seen by $Gal(K''/K)$. Therefore, Proposition \ref{normprop} can be obtained using only this quotient, and this is all that is required for the norm-combinatorics.
\end{rmk}

\section{Application: the Birational Section Conjecture}

In order to demonstrate how the main conjecture, and hence main result, can provide number theoretic consequences, we conclude with an application to the birational section conjecture. For details on this conjecture and the area of anabelian geometry it falls under, see e.g. \cite{koe3}, \cite{popmeta} and \cite{stix2012rational}. The details behind this application are surely well known to the experts, but the explicit nature of Conjecture 1 and Theorem 1 allow for a particularly direct argument.

\subsection{Basic statement of the birational section conjecture}

Recall that given a smooth, complete curve $X$ over a field $K$, there is a canonical exact sequence of Galois groups
\begin{equation}
\label{exact_sequence}
1 \rightarrow G_{\overline{K}(X)} \rightarrow G_{K(X)} \rightarrow G_K \rightarrow 1 
\end{equation}
where $K(X)$ is the function field of $X$ and $\overline{K}(X)=\overline{K} \otimes_K K(X)$. \\

Given any $a \in X(K)$, we can assign to it a `bouquet' of group-theoretic sections $s_a: G_K \rightarrow G_{K(X)}$. Indeed, let $v_a$ be the valuation on $K(X)$ corresponding to $a$, and $w$ the valuation on $\overline{K}(X)$ corresponding to a preimage of $a$ in $\overline{X}:=X \otimes_K \overline{K}$ (so $w$ extends $v$). Now let $I_w$ and $D_w$ denote the inertia and decomposition subgroups groups of $G_{K(X)}$ obtained by choosing an extension of $w$ to the separable closure $K(X)^{sep}$ of $K(X)$. Then Hilbert Decomposition Theory gives us a commutative diagram
\begin{center}
\begin{tikzpicture}[>=angle 90]
\matrix(d)[matrix of math nodes, row sep=3em, column sep=2em, text height=1.5ex, text depth=0.25ex]
{1 & G_{\overline{K}(X)} & G_{K(X)} & G_K & 1 \\
1 & I_w & D_w & G_w & 1\\};
\path[->, font=\scriptsize]
 (d-1-1) edge node[above]{} (d-1-2)
 (d-1-2) edge node{} (d-1-3)
 (d-1-3) edge node[above]{} (d-1-4)
 (d-1-4) edge node{} (d-1-5)
 (d-2-1) edge node[above]{} (d-2-2)
 (d-2-2) edge node{} (d-2-3)
 (d-2-3) edge node[above]{} (d-2-4)
 (d-2-4) edge node{} (d-2-5)
 (d-2-2) edge node{} (d-1-2)
 (d-2-3) edge node{} (d-1-3)
 (d-2-4) edge node[right]{$\simeq$} (d-1-4)
 ;
\end{tikzpicture}
\end{center}
with exact rows. Here $G_w$ denotes the Galois group of the residue field extension. It is known that the bottom row admits sections (see e.g. \cite{kpr}). Any choice of such induces a section $s_w$ of (\ref{exact_sequence}) such that $s(G_K) \subset D_w$. The set of sections $s_w$ obtained in this way is typically referred to as a `bouquet' of sections, and the members of the bouquet are said to lie over $a$. The different decomposition subgroups $D_w$ corresponding to different prolongations of $w$ to $K(X)^{sep}$ form a conjugacy class under $G_{K(X)}$. In a similar manner, if $v$ is a valuation which is trivial on $K$ and has residue field $K$, the same discussion shows that $v$ induces a `bouquet' of sections which are said to lie over $v$. We call such valuations {\bf $K$-valuations}. 

The \emph{birational anabelian section conjecture} of Grothendieck says that, for certain fields $K$, \emph{every} section of \ref{exact_sequence} lies over a unique $a \in X(K)$. This was proved in \cite{koe3} in the case where $K$ is a finite extension of $\mathbb{Q}_p$. 

\subsection{The pro-$p$ version of the birational section conjecture}

The exact sequence (5.1) does not necessarily remain exact upon taking maximal pro-$p$ quotients. Nevertheless, one can still consider sections of the canonical surjection
\begin{equation}
\label{p_exact_sequence}
G_{K(X)}(p) \twoheadrightarrow G_K(p),
\end{equation}
and show that a rational point $a \in X(K)$ gives rise a bouquet of sections $s$ with image inside a decomposition group of the associated valuation. These sections are once again said to lie over $a$. The pro-$p$ version of the birational section conjecture says that every section of (5.2) lies over such a $K$-rational point of $X$. We will show that Conjecture 1 implies the pro-$p$ birational section conjecture, assuming that $K$ is a finite extension of $\mathbb{Q}_p$ containing $\zeta_p$. Pop has shown in \cite{popmeta} that this conclusion holds already when considering a much smaller quotient, namely the maximal $\Z / p\Z$ elementary meta-abelian quotient. The pro-$p$ version follows from Pop's Theorem, but the proof we give here is independent. Both these quotient-variants of the birational section conjecture imply the standard birational section conjecture (see page 623 of \cite{popmeta}).

We also note that the generality of our main theorem means we can prove a statement for varieties, not just curves, with no additional complications. The birational analogue of the section conjecture in higher dimensions was proven by Stix in \cite{stixvarieties}. Here one finds that every section lies over a unique $K$-valuation. When $X$ is a curve, it is well known that such valuations correspond exactly to points. For higher dimensions, such valuations imply the \emph{existence} of a point, but there is no longer a bijective correspondence. In \cite{pop2017adic}, Pop generalized Stix's result to the metabelian setting, but again, our proof here is independent.

\begin{prop}\label{bsc} Assume Conjecture 1 holds, and suppose $X$ is a smooth, complete variety over $F$ of dimension $n>0$, where $F$ is a finite extension of $\mathbb{Q}_p$ containing $\zeta_p$. Then given any section $s$ of the projection
\begin{equation}
G_{F(X)}(p) \twoheadrightarrow G_F(p), \nonumber
\end{equation}
there is an $F$-valuation $w'$ on $F(X)$ with $s(G_{F}(p)) \subset D_{w'}$, the decomposition group of a prolongation of $w'$ to $F(X)(p)$. The valuation $w'$ and its prolongation to $F(X)(p)$ are uniquely determined by $s$.
\end{prop}
\begin{proof}
Let $s:G_F(p) \rightarrow G_{F(X)}(p)$ be a section, and let $K$ be the fixed field in $F(X)(p)$ of $s(G_F(p))$. Then $G_K(p) \simeq s(G_F(p)) \simeq G_F(p)$. By Conjecture 1, there is a finite extension $F'/ \mathbb{Q}_p$ (with a $p$-adic valuation $w$) and a valuation $v$ on $K$ satisfying the properties of the Conjecture. In particular, $Kv = F^\prime w$ and there are uniformizers $\pi$ and $\pi^\prime$ of $(K,v)$ and $(F^\prime ,w)$ respectively such that, for some natural number $e$,
both $v(p)=ev(\pi )$ and $w(p)=ew(\pi^\prime )$. We claim that in our case one can take $F^\prime = F$.

First, the valuation $v$ is the canonical $p$-henselian valuation on $K$ as the finite residue field is not $p$-closed,
and the restriction $u$ of $v$ to $F$ (which is relatively $p$-closed in $K$) is the canonical $p$-henselian valuation on $F$ and thus coincides with the $p$-adic valuation on $F$.

Now let $e_F=e_{F/\mathbb{Q}_p}, e_{F'}=e_{F'/\mathbb{Q}_p}$ denote the ramification indices of $F/\mathbb{Q}_p$ and $F'/\mathbb{Q}_p$ respectively, and similarly let $f_{F}, f_{F'}$ denote the inertia degrees. Since $F \subset K$, we know that $|F'w| = |Kv| \geq |Fu|$, and hence $f_F \geq f_{F'}$. Let $\lambda$ be a uniformizer of $F$, so that $u(p) = e_{F} u(\lambda)$. Since $v(\pi)$ is minimal positive in $\Gamma_v$ and $F \subset K$, we have $v(\pi) \leq v(\lambda)$, and thus the equality $v(p) = u(p)$ implies that $e_{F^\prime} = e \geq e_F$.

But now the isomorphism $G_{F'}(p) \simeq G_F(p)$ means the Demushkin invariants of $F$ and $F'$ match. In particular, the equality of $n$-invariants implies that $N:= [F':\mathbb{Q}_p] = [F:\mathbb{Q}_p]$. Thus
\begin{equation}
N = [F':\mathbb{Q}_p] = e_{F'}f_{F'} \geq e_F f_F = [F:\mathbb{Q}_p] = N, \nonumber
\end{equation}
and hence $e_F = e_{F'}, f_F = f_{F'}$. Therefore, $(F,u)$ satisfies all the conditions stated for $(F^\prime,w)$ in Conjecture 1
and hence we may take $F^\prime = F$ instead. Moreover, we may assume that $\pi = \lambda \in F$.

Then the restriction $v'$ of $v$ to $F(X)$ still has residue field $Fv$ and $v'(\pi)$ is a minimal positive element. If $H$ is the subgroup of $\Gamma_{v'}$ generated by $v'(\pi)$, we let $w'$ denote the valuation obtained from $v'$ with value group $\Gamma_{v'} / H$. The residue field of $w'$ contains $F$ and carries a discrete rank-1 valuation induced by $v'$ with the same residue field and ramification index as $(F,u)$, from which it follows that the residue field of $w'$ must be precisely $F$, since $(F,u)$ admits no proper immediate extension. As $(K,v)$ is $p$-henselian, so is the coarsening $\tilde{w}$ of $v$ with value group $\Gamma_v/H$; note that $w'$ is thus the restriction of $\tilde{w}$ to $F(X)$. Hence $K$ contains a $p$-henselization of $(F(X), w')$, that is, $s(G_F(p)) \subseteq D_{w'}$, the decomposition subgroup of $G_{F(X)}(p)$ with respect to the unique prolongation of $w'$ to $F(X)(p)$ which prolongs $\tilde{w}$.

To show uniqueness of $w'$, assume there is another such valuation $w''$. Then $w'$ and $w''$ give rise to two different $p$-henselian valuations on $K$ with non-$p$-closed residue field, which, by the $p$-variant of F. K. Schmidt's Theorem (Proposition 2.8 in \cite{koe2}), must be comparable. Hence so are $w'$ and $w''$, which is not possible as they are both trivial on $F$ with residue field precisely $F$.

Finally, to show uniqueness of the prolongation of $w'$ to $F(X)(p)$ containing $s(G_F(p))$ in its decompositions group, note that two distinct such prolongations would restrict to two distinct $p$-henselian, hence comparable, valuations on $K$, both prolonging $w'$. But $K/F(X)$ is algebraic and hence no two prolongations of $w'$ to $K$ can be comparable.
\end{proof}

\begin{cor}
Assume Conjecture 1 holds, and suppose $X$ is a smooth, complete variety over $F$, where $F$ is a finite extension of $\mathbb{Q}_p$ containing $\zeta_p$. Then there is exists a section of $G_{F(X)}(p) \twoheadrightarrow G_F(p)$ if and only if $X(F) \not = \emptyset$.
\end{cor}
\begin{proof}
The non-trivial direction is to go from a section to a point. Assuming a section exists, Proposition \ref{bsc} shows that one can construct a $F$-valuation $w'$ on $F(X)$, i.e. a valuation which is trivial on $F$ and has residue field $F(X)v = F$. If $\mathcal{O}_{w'}$ is the corresponding valuation ring with maximal ideal $\mathcal{M}_{w'}$, then triviality of the valuation on $F$ implies that we have a commutative diagram
\begin{center}
\begin{tikzpicture}
  \matrix (m) [matrix of math nodes,row sep=3em,column sep=4em,minimum width=2em]
  {
     Spec(F(X)) & X \\
     Spec(\mathcal{O}_{w'}) & Spec(F) \\};
  \path[-stealth]
    (m-1-1) edge node [left] {} (m-2-1)
            edge [] node [below] {} (m-1-2)
    (m-2-1.east|-m-2-2) edge node [below] {}
            node [above] {} (m-2-2)
    (m-1-2) edge node [right] {} (m-2-2);
\end{tikzpicture}
\end{center}
Since $X$ is complete, the valuative criterion of properness implies that there is a unique map $i:Spec(\mathcal{O}_{w'}) \rightarrow X$ for which the diagram still commutes. The image of $\mathcal{M}_{w'}$ under $i$ is a closed point of $X$ with residue field $K(X)v = F$, and is therefore precisely an $F$-rational point of $X$, as desired.
\end{proof}

\begin{cor}
Assume Conjecture 1 holds, and suppose $X$ is a smooth, complete \emph{curve} over $F$, where $F$ is a finite extension of $\mathbb{Q}_p$ containing $\zeta_p$. Then every section of the projection $G_{F(X)}(p) \twoheadrightarrow G_F(p)$ lies over a unique $F$-rational point $a \in X(F)$.
\end{cor}
\begin{proof}
This follows from the classical result that, for curves, all $K$-valuations come from $K$-rational points.
\end{proof}

\noindent The main theorem therefore yields the following unconditional result:

\begin{cor} 
Suppose $X$ is a smooth, complete variety over $\Q$. Then any section of the projection 
\begin{equation}
G_{\Q(X)}(2) \twoheadrightarrow G_{\Q}(2) \nonumber 
\end{equation}
lies above a unique $\Q$-valuation $v$, which corresponds to a $\Q$-rational point if $X$ is a curve. In both cases, the existence of a section implies that $X(\Q) \not = \emptyset$.
\end{cor}

\begin{rmk}
As noted before, one can actually obtain the same conclusion using just the maximal $\Z /2\Z$ elementary meta-abelian quotients. A key point here is that $p$-henselianity can be detected purely from degree $p$ extensions by Fact \ref{fact:phenselian}, which allows the arguments in Proposition \ref{bsc} to go through.
\end{rmk}

 \appendix
\section{}

\subsection{Description of norm-combinatorics algorithm}

In these appendices we will detail the computations missing from the main text. While these computations were initially verified by hand by way of a form of `yoga' of norm-combinatorics, the fundamentally algorithmic nature of the procedure allowed us to verify and streamline the process using a computer program. The code, written in Python, is registered on Zenodo:\\

\url{https://zenodo.org/badge/latestdoi/63322560}\\

\noindent The code repository can be checked out and run on any computer with Python installed\footnote{The only non-standard package required is the module Sympy, which can be found at \url{http://www.sympy.org/en/index.html}}. The script `GenerateProof.py' will produce several text-files, which comprise a fully fleshed out proof of Proposition \ref{propxy}. For the reader not interested in running code, the text-files `ProofOfPropositon.txt' and `DetermineValuesCaseN.txt' (N=1,..8)  can be downloaded directly and viewed on the GITHub page. The file `ProofOfProposition' is the main proof file, with the other files containing details for how to pin down values of $1+ax$ for various $a$.

Let us first describe what seems to be happening at a more abstract level, before describing the explicit algorithm used. Suppose we are given the square values of $x$ and $1+x$ and want to determine the value of $1+nx$ for some integer $n$. Note that for \emph{any} integer $m$, we have
\begin{equation}
\label{decomp1}
1+nx = (1+mx) + (n-m)x.
\end{equation}
We therefore define the following subset of $K^{\times}/(K^{\times})^2$:
\begin{equation}
C^n_{m}(x) = N(-nx) \cap (1+mx)N(-(1+mx)(n-m)x). \nonumber
\end{equation}
Equation \ref{decomp1} implies that $1+nx \in C^n_{m}(x)$. As this is true for every $m$, we get
\begin{equation}
1+nx \in \bigcap_{m \in \mathbb{Z}} C^n_{m}(x),
\end{equation}
which is a coset of an intersection of norm groups. These intersections, for each $n,m$, form a system of dependencies between the square values of $1+nx$ for every $n \in \mathbb{Z}$. In particular, because every expression $1+nx$ does \emph{have} a square value, these intersections can never be empty. The algorithm we implement amounts to showing that, given the norm lattice of $K$, there is only one possible square value of $1+nx$ which is consistent with these dependencies. The `correct' value, which corresponds to that expected\footnote{Recall that, since a 2-adic integer $z$ is a square in $\mathbb{Z}_2$ iff $z \equiv 1$ mod 8, the square class of expressions $1+nx$ in $\mathbb{Z}_2$ can always be deduced by considering the valuation of $nx$.} from the existence of a dyadic valuation on $K$, produces an intersection of size 1, while any other choice produces an empty intersection. It is tempting to speculate that the embedding of the norm lattice into the right framework may allow one to interpret these intersection-dependencies as a system of linear equations with a unique solution. The resemblance between expressions such as $1-xy$ with arithmetic in the (2-truncated rational) big Witt vectors $W(K)$ over $K$, suggests one candidate for this framework, but not one we are able to substantiate further.
\\
\\
The actual algorithm is now simple to describe. As input is given the square value of $x, 1+x$ and an integer $n$: the desired output is the square value of $1+nx$.
\\
\\
\textbf{Step 1: }Compute $C^n_{1}(x)$. As this is a coset of an intersection of two norm groups, its size is at most 4.
\\
\textbf{Step 2: }Let $y$ be one of the 4 possible values and assume $1+nx \sim y$.
\\
\textbf{Step 3: }Pick some $m \not = n$. Compute the intersection $C^m_{1}(x) \cap C^m_{n}(x)$ using the assumption from Step 2. The value of $1+mx$ lies in their intersection. If this intersection is empty, we know that the assumption in Step 2 was false: go back to Step 2, pick a different choice of $1+nx$, and compute the intersection for $1+mx$ again. Repeat until a non-empty intersection is obtained.
\\
\textbf{Step 4: }The non-empty intersection is again, at most size 4, and we simply repeat the above procedure. Assume $1+mx \sim z$ for one these 4 choices, pick some $m' \not = n,m$, and compute the intersection $C^{m'}_{1}(x) \cap C^{m'}_{n}(x) \cap C^{m'}_{m}(x)$. If this intersection is empty, reject the choice $z$. Pick a new choice and repeat.
\\
\\
\noindent The key point is that, if the initial choice $y$ was bad, then this procedure terminates (i.e. produces an empty intersection) after a very short number of steps: typically only 2 or 3 rounds suffice. Furthermore, given $n=\pm2, \pm4, \pm5$, it is always sufficient to pick $m$ in the range $n-5<m<n+5$. In practice therefore, if the initial choice $y$ does not produce a contradiction after 4 rounds of the above algorithm, then one simply moves to the next case. In the next section we will explicitly show how this procedure terminates in all cases and gives the desired output.

\begin{rmk}
Repeated use is made of the fact that we know the square class in $k$ of every integer, since it must be the same as in $\Q$ by Proposition \ref{normprop}. As the computations only involve expressions in $x$ with integer coefficients, this suffices. \\
\end{rmk}

Verifying Proposition \ref{propxy} is now easy. One simply computes all the decompositions $D_a(x,y)$ for $a = \pm 1, \pm 5$. The square values of $1+ax$ and $1+ay$ are known by virtue of the above algorithm. Intersecting all the constraints obtained in this way turns out to be enough to pin down the value of $1-xy$ precisely, and it can be seen to be in $N(5)$ in all cases.

\subsection{Proof of Lemma \ref{lemma24} and Proposition \ref{propxy}}

Let us now prove Lemma \ref{lemma24} using the algorithm described above. We want to show that for $x \in \val_1$, $1+ax \in N(5)$ for $a = \pm 1, \pm 2, \pm 3, \pm 4, \pm 5, \pm 1/5$. The constraints $C^n_m(x)$ will take care of all values of $a$ except $a = \pm 1/5$. For these, we will use the fact that we will have already determined the values of $1+mx$ for certain integers $m$. Indeed, note that to determine $1+x/n$ for $n= \pm 5$, it suffices to determine $n+x$. Now note that if $1+mx \sim k$, then $1+mx = kb^2$ for some $b \in K$. Therefore, we find by simple arithmetic that
\begin{equation}
x = -\frac{1}{m} + mkc^2 \nonumber
\end{equation}
for some $c \in K$, and hence
\begin{eqnarray}
\label{decomp2}
n+x &=& \Big( n-\frac{1}{m} \Big) + mkc^2 \nonumber \\
&=& \frac{nm-1}{m} + mkc^2. 
\end{eqnarray}
Define the following subset of $K^{\times}/(K^{\times})^2$:
\begin{equation}
F^n_{m}(x) = nN(-nx) \cap m(nm-1)N(-(nm-1)(1+mx)) \nonumber
\end{equation}
By equation (\ref{decomp2}), we have $n+x \in F^n_m(x)$, and we can compute these sets using our knowledge about $1+mx$. Intersecting several such subsets corresponding to values $m = \pm 1, \pm 2, \pm 3, \pm4, \pm5$ provides a sufficiently stringent constraint on $n+x$ to pin down its value uniquely in all cases. The values $n-x$ can be obtained via $n-x=-(-n+x)$, using the same constraints.

We remind the reader that in $\Q$, and hence also in $K$, we have $3 \sim -5$ and $7 \sim -1$. This allows one to easily reduce the value of any integer appearing in the proof to one of $\pm 1, \pm 2, \pm5$ or $\pm 10$.

Finally, by convention we set $N(1) = \{\pm 1, \pm 2, \pm 5, \pm 10\}$.

\begin{proof}
We proceed by cases, based on the square classes of $x$ and $1+x$. Throughout we use freely Proposition \ref{normprop} to determine intersections of various norm groups. The constraints from the decompositions $C^n_m(x)$ are made explicit in the first few cases, and then quoted without comment.

Note also that $x \in \val_1$ implies $x \sim \pm2, \pm 10$. If $x \sim 2$, then $1+x \in N(5) \cap N(-2) = \{1,-5\}$, and similarly when $x \sim -2$ or $\pm10$. Therefore the cases considered below are indeed exhaustive.\\

\noindent \boxed{\textbf{Case 1: } x \sim 2, 1+x \sim 1}\\ 
\begin{itemize}
\item[(1)] $1+2x \in C^2_1(x) = N(-2x) \cap (1+x)N(-(1+x)x) = N(-1) \cap N(-2) = \{1,2\}$. Hence first-order constraints did not suffice to pin down the value uniquely. Assume therefore that $1+2x \sim 2$. Then we have $1+3x \in C^3_1(x) = N(-3 \cdot 2) \cap 1 \cdot N(-1 \cdot 1 \cdot (3-1) \cdot 2) = N(10) \cap N(-1) = \{1,10\}$. The constraint from $C^3_2(x)$ gives nothing new in this case, so we have $1+3x \sim 1$ or $10$. We consider both options in turn.\\

\begin{itemize}
\setlength{\itemindent}{0.18in}
\item Suppose $1+3x \sim 1$. Then $C^5_1(x) = N(-10) \cap N(-2) = \{1, -5\}$, and $C^5_2(x) = N(-10) \cap 2N(5) = \{-2, 10\}$. Hence $1+5x \in C^5_1(x) \cap C^5_2(x)  = \varnothing$: contradiction.\\
\item Suppose $1+3x \sim 10$. The same computation as above gives a contradiction once more.\footnote{Note that in this particular case, the assumption about $1+2x$ alone was enough to get a contradiction, but this is not always the case. Typically the extra constraint from $1+3x$ is required.}\\
\end{itemize}
Hence our assumption that $1+2x \sim 2$ must have been false, implying that $1+2x \sim 1$.\\

\item[(2)] We have showed that $1+2x \sim 1$. Thus $1+3x \in C^3_1(x) \cap C^3_2(x) = N(10) \cap N(-1) \cap N(-2) = \{1\}$.\\
\item[(3)] Similarly, we can now compute that $1+4x \in C^4_1(x) \cap C^4_2(x) \cap C^4_3(x) = \{1\}$.\\
\item[(4)] $1+5x \in C^5_1(x) \cap C^5_2(x) \cap C^5_3(x) \cap C^5_4(x) = \{1\}$.\\
\item[(5)] Intersecting all available constraints gives $1-2x \in \{1\}$.\\
\item[(6)] Intersecting all available constraints gives $1-3x \in \{1\}$.\\
\item[(7)] Intersecting all available constraints gives $1-4x \in \{1\}$.\\
\item[(8)] Intersecting all available constraints gives $1-5x \in \{1\}$.\\
\item[(9)] Intersecting all available constraints gives $1-x \in \{1\}$.\\
\item[(10)] We have $5+x \in F^5_1(x) = 5N(-10) \cap N(-1) = \{2,5\}$. Also, $5+x \in F^5_3(x) = 5N(-10) \cap 3(15-1)N(-(15-1)\cdot 1) = 5N(-10) \cap 10N(-2) = \{5,-10\}$, using that $14 \sim-2$. Hence $5+x \in \{2,5\} \cap \{5,10\} = \{5\}$.\\
\item[(11)] We have $5-x \in -F^{-5}_1(x) \cap -F^{-5}_2(x) = \{5, 2\} \cap \{-2, 5\} = \{5\}$.\\ \\
\end{itemize}

\noindent \boxed{\textbf{Case 2: } x \sim 2, 1+x \sim -5}\\

\begin{itemize}
\item[(1)] $1+2x \in C^2_1(x) = \{2,5\}$.\\

\begin{itemize}
\setlength{\itemindent}{0.18in}
\item Assume $1+2x \sim 2$. Then $1+3x \in C^3_1(x) \cap C^3_2(x) = \{1\}$. But then $1+5x \in C^5_1(x) \cap C^5_2(x) \cap C^5_3(x) = \varnothing$: contradiction.\\
\end{itemize}
Hence $1+2x \sim 5$.\\

\item[(2)] All further values, including fractional ones, are now obtained immediately by successively constraining, in order, the remaining values. This gives: $1+3x \sim -1, 1+4x \sim 1, 1+5x \sim -5, 1-2x \sim 5, 1-3x \sim -5, 1-4x \sim 1, 1-5x \sim -1, 1-x \sim -1, 5+x \sim -1, 5-x \sim -5$.\\
\end{itemize}


\noindent \boxed{\textbf{Case 3: } x \sim -2, 1+x \sim 1} \\
\begin{itemize}
\item[(1)] $1+2x \in C^2_1(x) = N(2) = \{1,-1,2,-2\}$.\\

\begin{itemize}
\setlength{\itemindent}{0.18in}

\item Assume $1+2x \sim -1$. Then $1+3x \in C^3_1(x) \cap C^3_2(x) = \{10,-2\}$. \\
\begin{itemize}
\setlength{\itemindent}{0.25in}
\item If $1+3x \sim 10$, then $1-4x \in C^{-4}_1(x) \cap C^{-4}_2(x) \cap C^{-4}_3(x) = \varnothing$.
\item If $1+3x \sim -2$, then $1-2x \in C^{-2}_1(x) \cap C^{-2}_2(x) \cap C^{-2}_3(x) = \varnothing$.\\
\end{itemize}
\item Assume $1+2x \sim 2$. Then $1+3x \in C^3_1(x) \cap C^3_2(x) = N(-10) = \{1,-2,-5,10\}$. Irrespective of what choice one picks for $1+3x$, a contradiction is obtained by considering $1-3x$.\\
\item Assume $1+2x \sim -2$. Then $1+3x \in C^3_1(x) \cap C^3_2(x) = \{-5, -2\}$. Irrespective of what choice one picks for $1+3x$, a contradiction is obtained by considering $1+5x$.\\
\end{itemize}
Hence it must be that $1+2x \sim 1$.\\

\item[(2)] $1+3x \in C^3_1(x) \cap C^3_2(x) = \{1,-2\}$. \\
\begin{itemize}
\setlength{\itemindent}{0.18in}
\item Assume that $1+3x \sim -2$. Then $1+4x \in C^4_1(x) \cap C^4_2(x) \cap C^4_3(x) = \{-2\}$, so $1+4x \sim -2$. But now a contradiction is obtained by considering $1+5x$.\\
\end{itemize}
Hence it must be that $1+3x \sim 1$.\\

\item[(3)] $1+4x \in C^4_1(x) \cap C^4_2(x) \cap C^4_3(x) = \{1, -2\}$. If we assume $1+4x \sim -2$, a contradiction is obtained by considering $1+5x$. Hence $1+4x \sim 1$.\\

\item[(4)] All further values, including fractional values, are now obtained by successively constraining the values remaining, in order. This gives $1+5x \sim 1, 1-2x \sim 1, 1-3x \sim 1, 1-4x \sim 1, 1-5x \sim 1, 1-x \sim 1, 5+x \sim 5, 5-x \sim 5$.\\ \\
\end{itemize}

\noindent \boxed{\textbf{Case 4: }x \sim -2, 1+x \sim -1}\\

\begin{itemize}
\item[(1)] $1+2x \in C^2_1(x) = \{-1,-2,5,10\}$.\\

\begin{itemize}
\setlength{\itemindent}{0.18in}
\item Assume $1+2x \sim -1$. Then $1+3x \in C^3_1(x) \cap C^3_2(x) = \{-2\}$. But now a contradiction is obtained by considering $1-4x$.\\
\item Assume $1+2x \sim -2$. Then $1+3x \in C^3_1(x) \cap C^3_2(x) = \{-2, -5\}$.\\
\begin{itemize}
\setlength{\itemindent}{0.25in}
\item If $1+3x \sim -2$, obtain a contradiction by considering $1-3x$.
\item If $1+3x \sim -5$, obtain a contradiction by considering $1+4x$.\\
\end{itemize}
\item Assume $1+2x \sim 10$.  Then $1+3x \in C^3_1(x) \cap C^3_2(x) = \{-2\}$. But now a contradiction is obtained by considering $1-5x$.\\
\end{itemize}
Hence we must have $1+2x \sim 5$.

\item[(2)] $1+3x \in C^3_1(x) \cap C^3_2(x) = \{-2,-5\}$.\\
\begin{itemize}
\setlength{\itemindent}{0.18in}
\item Assume $1+3x \sim -2$. Then $1+4x \in C^4_1(x) \cap C^4_2(x) \cap C^4_3(x) = \{-1\}$. Now obtain a contradiction by considering $1-3x$.\\
\end{itemize}
Hence $1+3x \sim -5$.\\

\item[(3)] All further values, including fractional values, are now obtained by successively constraining the values remaining, in order. This gives: $1+4x \sim 1, 1+5x \sim -1, 1-2x \sim 5, 1-3x \sim -1, 1-4x \sim 1, 1-5x \sim -5, 1-x \sim -5, 5+x \sim -5, 5-x \sim -1$.\\ \\
\end{itemize}

\noindent \boxed{\textbf{Case 5: }x \sim 10, 1+x \sim 1}\\

\begin{itemize}
\item[(1)] $1+2x \in C^2_1(x) = \{1,-2\}$. \\

\begin{itemize}
\setlength{\itemindent}{0.18in}
\item Assume $1+2x \sim -2$. Then $1+3x \in C^3_1(x) \cap C^3_2(x) = \{-2\}$. Now a contradiction is obtained by considering $1-2x$. \\
\end{itemize}
Hence $1+2x \sim 1$.\\

\item[(2)] $1+3x \in C^3_1(x) \cap C^3_2(x) = \{1,-2\}$. \\
\begin{itemize}
\setlength{\itemindent}{0.18in}
\item Assume $1+3x \sim -2$. Then $1+4x \in C^4_1(x) \cap C^4_2(x) \cap C^4_3(x) = \{-2\}$. Now a contradiction is obtained by considering $1+5x$.\\
\end{itemize}
Hence $1+3x \sim 1$.\\

\item[(3)] $1+4x \in C^4_1(x) \cap C^4_2(x) \cap C^4_3(x) = \{1,-2\}$.\\
\begin{itemize}
\setlength{\itemindent}{0.18in}
\item Assume $1+4x \sim -2$. Then a contradiction is obtained by considering $1+5x$.\\
\end{itemize}
Hence $1+4x \sim 1$.\\

\item[(4)] All further values, including fractional values, are now obtained by successively constraining the values remaining, in order. This gives: $1+5x \sim 1, 1-2x \sim 1, 1-3x \sim 1, 1-4x \sim 1, 1-5x \sim 1, 1-x \sim 1, 5+x \sim 5, 5-x \sim 5$.\\ \\
\end{itemize}

\noindent \boxed{\textbf{Case 6: }x \sim 10, 1+x \sim -5} \\

\begin{itemize}
\item[(1)] $1+2x \in C^2_1(x) = \{5, -10\}$. \\

\begin{itemize}
\setlength{\itemindent}{0.18in}
\item Assume $1+2x \sim -10$. Then $1+3x \in C^3_1(x) \cap C^3_2(x) = \{1,-1,2,-2\}$.\\
\begin{itemize}
\setlength{\itemindent}{0.25in}
\item If $1+3x \sim 1$, obtain a contradiction by considering $1-2x$.
\item If $1+3x \sim 2$, obtain a contradiction by considering $1-5x$.
\item If $1+3x \sim -1$, obtain a contradiction by considering $1+5x$.
\item If $1+3x \sim -2$, obtain a contradiction by considering $1+5x$.\\
\end{itemize}
\end{itemize}
Hence $1+2x \sim 5$.\\

\item[(2)] $1+3x \in C^3_1(x) \cap C^3_2(x) = \{-1,-2\}$.\\
\begin{itemize}
\setlength{\itemindent}{0.18in}
\item Assume $1+3x \sim -2$. Then $1+4x \in C^4_1(x) \cap C^4_2(x) \cap C^4_3(x) = \{10\}$. Now obtain a contradiction by considering $1+5x$.\\
\end{itemize}
Hence $1+3x \sim -1$.\\

\item[(3)] $1+4x \in C^4_1(x) \cap C^4_2(x) \cap C^4_3(x) = \{1,10\}$.\\
\begin{itemize}
\setlength{\itemindent}{0.18in}
\item Assume $1+4x \sim 10$.  Then one obtains a contradiction by considering $1+5x$.\\
\end{itemize}
Hence $1+4x \sim 1$.\\

\item[(4)] All further values, including fractional values, are now obtained by successively constraining the values remaining, in order. This gives: $1+5x \sim -5, 1-2x \sim 5, 1-3x \sim -5, 1-4x \sim 1, 1-5x \sim -1, 1-x \sim -1, 5+x \sim -1, 5-x \sim -5$.\\ \\
\end{itemize}

\noindent \boxed{\textbf{Case 7: } x \sim -10, 1+x \sim 1}\\

\begin{itemize}
\item[(1)] $1+2x \in C^2_1(x) = \{1,-1\}$. \\

\begin{itemize}
\setlength{\itemindent}{0.18in}
\item Assume $1+2x \sim -1$. Then a contradiction is obtained by considering $1+3x$.\\
\end{itemize}
Hence $1+2x \sim 1$.\\

\item[(2)] All further values, including fractional values, are now obtained by successively constraining the values remaining, in order. This gives: $1+3x \sim 1, 1+4x \sim 1, 1+5x \sim 1, 1-2x \sim 1, 1-3x \sim 1, 1-4x \sim 1, 1-5x \sim 1, 1-x \sim 1, 5+x \sim 5, 5-x \sim 5$.\\ \\
\end{itemize}

\noindent \boxed{\textbf{Case 8: } x \sim -10, 1+x \sim -1}\\

\begin{itemize}
\item[(1)] $1+2x \in C^2_1(x) = \{5,-1\}$.\\

\begin{itemize}
\setlength{\itemindent}{0.18in}
\item Assume $1+2x \sim -1$. Then $1+3x \in C^3_1(x) \cap C^3_2(x) = \{2\}$. Now a contradiction is obtained by considering $1+4x$.\\
\end{itemize}
Hence $1+2x \sim 5$.\\

\item[(2)] All further values, including fractional values, are now obtained by successively constraining the values remaining, in order. This gives: $1+3x \sim -5, 1+4x \sim 1, 1+5x \sim -1, 1-2x \sim 5, 1-3x \sim -1, 1-4x \sim 1, 1-5x \sim -5, 1-x \sim -5, 5+x \sim -5, 5-x \sim -1$.\\
\end{itemize}

\noindent This completes the proof.
\end{proof}

The results of these computations that are used in the proof of Proposition \ref{propxy} are summarized in Table A1:

\begin{table}[h] 
\caption{Stability of $\val_1$ under $N(5)$} 
\centering 
\begin{tabular}{l  c || c cccc} 
\hline\hline 
$x$ & $1+x$ & $1-x$ & $1+5x$ & $1-5x$ & $1+5^{-1}x$ & $1-5^{-1}x$
\\ [0.5ex] 
\hline 
2 & 1 & 1 & 1 & 1 & 1  & 1 \\
2 & -5 & -1 & -5 & -1 & -5 & -1 \\
-2 & 1 & 1 & 1 & 1 & 1 & 1 \\
-2 & -1 & -5 & -1 & -5 & -1 & -5 \\
10 & 1 & 1 & 1 & 1 & 1 & 1 \\
10 & -5 & -1 & -5 & -1 & -5 & -1 \\
-10 & 1 & 1 & 1 & 1 & 1 & 1 \\
-10 & -1 & -5 & -1 & -5 & -1 & -5 \\
 
\hline 
\end{tabular} 
\label{tab:PPer}
\caption{For each possible combination of $x$ and $1+x$, $x \in \val_1$, the corresponding row gives the square values of $1+ax$ for different values of $a$.}
\end{table}

We are now ready to prove Proposition \ref{propxy}. The proof strategy was explained in the main text. Table A1 gives us all the information required to produce strong constraints on the possible values of $1-xy$ for each pair $x,y \in \val_1$. We specify sufficient decompositions $D_a(x,y), D_a(y,x)$ in each such case to force $1-xy \in N(5)$.

Before proceeding, observe that any $x \in \val_1$ is of the form $2ab^2$ where $a \in \{\pm1, \pm 5 \}$. Since for any such $a$ and any other $y \in \val_1$, $ay \in \val_1$, we can without loss of generality always assume that $x \sim 2$. By the symmetry of $x$ and $y$ in $1-xy$, the case $x \sim y \sim 2, 1+x \sim -5, 1+y \sim  1$ follows from the case defined by swapping $x$ and $y$. The cases checked below are therefore exhaustive.

\begin{proof}(Proposition \ref{propxy})
By the above discussion, we only need to check the following cases:
\\

\boxed{x \sim 2, 1+x \sim 1, y \sim 2, 1+y \sim 1}\,:\\
\begin{itemize}
\setlength{\itemindent}{0.18in}
\item[] $1-xy \in D_1(x,y) \cap D_{-1}(x,y) \cap D_5(x,y) = \{1\} \subset N(5)$.\\
\end{itemize}

\boxed{x \sim 2, 1+x \sim 1, y \sim 2, 1+y \sim -5}\,:\\
\begin{itemize}
\setlength{\itemindent}{0.18in}
\item[] $1-xy \in D_1(x,y) \cap D_{-1}(x,y) \cap D_5(x,y) = \{1\} \subset N(5)$.\\
\end{itemize}

\boxed{x \sim 2, 1+x \sim 1, y \sim -2, 1+y \sim 1}\,:\\
\begin{itemize}
\setlength{\itemindent}{0.18in}
\item[] $1-xy \in D_1(x,y) \cap D_5(x,y) = \{1\} \subset N(5)$.\\
\end{itemize}

\boxed{x \sim 2, 1+x \sim 1, y \sim -2, 1+y \sim -1}\,: \\
\begin{itemize}
\setlength{\itemindent}{0.18in}
\item[] $1-xy \in D_1(x,y) \cap D_{-1}(x,y) = \{1\} \subset N(5)$.\\
\end{itemize}

\boxed{x \sim 2, 1+x \sim 1, y \sim 10, 1+y \sim 1}\,:\\
\begin{itemize}
\setlength{\itemindent}{0.18in}
\item[] $1-xy \in D_1(x,y)  = \{1,-1\} \subset N(5)$.\\
\end{itemize}

\boxed{x \sim 2, 1+x \sim 1, y \sim 10, 1+y \sim -5}\,:\\
\begin{itemize}
\setlength{\itemindent}{0.18in}
\item[] $1-xy \in D_1(x,y)  = \{1,-5\} \subset N(5)$.\\
\end{itemize}

\boxed{x \sim 2, 1+x \sim 1, y \sim -10, 1+y \sim 1}\,:\\
\begin{itemize}
\setlength{\itemindent}{0.18in}
\item[] $1-xy \in D_1(x,y) \cap D_{-1}(x,y)  = \{1\} \subset N(5)$.\\
\end{itemize}

\boxed{x \sim 2, 1+x \sim 1, y \sim -10, 1+y \sim -1}\,: \\
\begin{itemize}
\setlength{\itemindent}{0.18in}
\item[] $1-xy \in D_1(x,y) \cap D_5(x,y)  = \{1\} \subset N(5)$.\\
\end{itemize}

\boxed{x \sim 2, 1+x \sim -5, y \sim 2, 1+y \sim -5}\,:\\
\begin{itemize}
\setlength{\itemindent}{0.18in}
\item[] $1-xy \in D_1(x,y) \cap D_{-1}(x,y) \cap D_5(x,y)  = \{5\} \subset N(5)$.\\
\end{itemize}

\boxed{x \sim 2, 1+x \sim -5, y \sim -2, 1+y \sim 1}\,: \\
\begin{itemize}
\setlength{\itemindent}{0.18in}
\item[] $1-xy \in D_1(x,y) \cap D_{-1}(x,y)   = \{1\} \subset N(5)$.\\
\end{itemize}

\boxed{x \sim 2, 1+x \sim -5, y \sim -2, 1+y \sim -1}\,:\\
\begin{itemize}
\setlength{\itemindent}{0.18in}
\item[] $1-xy \in D_1(x,y)  \cap D_5(x,y)  = \{5\} \subset N(5)$.\\
\end{itemize}

\boxed{x \sim 2, 1+x \sim -5, y \sim 10, 1+y \sim 1}\,:\\
\begin{itemize}
\setlength{\itemindent}{0.18in}
\item []$1-xy \in D_1(x,y)  = \{1,-5\} \subset N(5)$.\\
\end{itemize}

\boxed{x \sim 2, 1+x \sim -5, y \sim 10, 1+y \sim -5}\,:\\
\begin{itemize}
\setlength{\itemindent}{0.18in}
\item[] $1-xy \in D_1(x,y) = \{5,-5\} \subset N(5)$.\\
\end{itemize}

\boxed{x \sim 2, 1+x \sim -5, y \sim -10, 1+y \sim 1}\,: \\
\begin{itemize}
\setlength{\itemindent}{0.18in}
\item[] $1-xy \in D_1(x,y) \cap D_5(x,y)  = \{1\} \subset N(5)$.\\
\end{itemize}

\boxed{x \sim 2, 1+x \sim -5, y \sim -10, 1+y \sim -1}\,: \\
\begin{itemize}
\setlength{\itemindent}{0.18in}
\item[] $1-xy \in D_1(x,y) \cap D_{-1}(x,y) = \{5\} \subset N(5)$.\\
\end{itemize}

\noindent In all cases $1-xy \in N(5)$, completing the proof.
\end{proof}

\section{Ruling out the exceptional case}

In this section we aim to prove Proposition \ref{caseb}, i.e. that the exceptional Case B of Lemma \ref{keylemma} cannot occur. 

\subsection{Preliminary observations about Case B.}
\label{sec:c0}

Recall that $q_2(K) = 3$, and we have independent elements $-1$ and $2$ in $\sq$, with each norm-group having dimension 2. Therefore we can always pick $c \in K$ independent of $-1$ and $2$, with $c \in N_K(-1)$. Hence $\sq = \{1,-1,2,-2,c,-c,2c,-2c\}$.

Thus we automatically have $N_K(-1) = \langle 2,c \rangle$, and $N_K(2) = \langle -1, 2 \rangle$. Consider $N_K(c)$. Trivially we have $-c \in N_K(c)$. If $2 \in N_K(c)$, then $c \in N_K(2)$ by reciprocity. As this is not the case, we have $2 \not \in N_K(c)$. This rules out $\pm 2c$ being in $N_K(c)$ as well, which implies we must have $N_K(c) = \langle -1, c \rangle$. Using the fact that the norm groups need to be entirely distinct, since $K$ is Demushkin, one can proceed similarly to determine the norm groups of $K$. We end up with the following `lattice':\\
\begin{eqnarray}
N_K(-1) &=& \langle 2,c \rangle \nonumber \\
N_K(2) &=&\langle -1,2 \rangle \nonumber \\
N_K(-2)& =& \langle 2, -c \rangle \nonumber \\
N_K(c) &=& \langle -1,c \rangle \nonumber \\
N_K(-c) &=& \langle c, -2 \rangle \nonumber \\
N_K(2c) &=& \langle -1, 2c \rangle \nonumber \\
N_K(-2c) &=& \langle -c, -2 \rangle \nonumber
\end{eqnarray}


\subsection{Motivation}

Because the proof which follows, presented without motivation, would no doubt appear extremely ad hoc to the reader, we also discuss some of the reasoning that led to its discovery.

By refining the analysis of Case B carried out in Lemma \ref{keylemma}, one can show, in the case where $\text{dim}_{\mathbb{F}_p} k^{\times}/(k^{\times})^2 = 3$, that $k$ admits both an ordering and a $p$-adic valuation with $p \equiv 5 \, (8)$ and a residue field that is not $2$-closed.\footnote{To see why, note that in this case one has $c \not \in N_k(2)$ but, being positive, $c \in N_k^r(2)$ with $k^r$ the real closure. Taking an inverse limit over the valuations/orderings for which $c \not \in N_{k^h}(2)$ produces a valuation on $k$ which is either dyadic, putting us back in Case A, or is as described here. It cannot be another ordering, by construction.} Let $k^h$ denote the henselization of $k$ with respect to this $p$-adic valuation. In this case, $2$ is a square in the real closure of $k$ but not in $k$, so we must have that $2 \not \in (k^h)^2$ and so, by henselianity, $2 \not \in (k^hv_p)^2$. Thus $k^h(\sqrt{2})$ is the unique unramified quadratic extension of $k^h$, and so $\val(N_k(2))$ defines the $p$-adic valuation ring in $k^h$, and hence also in $k$. The `yoga' of norm-combinatorics suggests that $\val(N_K(2))$ should then also define a valuation on $K$, again with residue field not $2$-closed. This would provide a contradiction, using the main results from \cite{koe2}.

The process of attempting to carry through the norm-combinatorics in this case lead the authors to a contradiction in any situation where $3 \sim 1$ or $3 \sim 2$, as in sub-case $B_1$. Sub-case $B_2$ was found to be ruled out in a similar manner. The contradiction can be understood, in both sub-cases, as showing that certain elements which necessarily exist in $K$, due to $K$ being Demushkin, are incompatible with the norm group information inferred from $k$ (e.g. the square value of $3$). These are elements which cannot exist in $k$ because of the existence of the ordering, and since, `morally' speaking, $K$ should inherit the valuation-theoretic information of $k$, these elements should therefore not exist in $K$ either.

\subsection{Proof of Proposition \ref{caseb}}

We deal with the two sub-cases separately.

\subsubsection*{Case $B_1$.}
\label{sec:c1}

We need two preliminary lemmas, which show that certain arithmetic features of $K$ are in this case independent of whether $3 \sim 1$ or $2$. Note that the norm groups in what follows are computed using the norm lattice presented in the previous section.
\\

\begin{lemma}
Assume $3 \sim 2$. Let $x \in K$ be such that $x \sim c, 1+x \sim -2$, and $y \in K$ be such that $y \sim -2c, 1+y \sim -1$. Then it is necessarily the case that $1-x \sim 1$ and $1-y \sim 1$.
\end{lemma}
\begin{proof}
Let us first show the claim about $x$.\\
\begin{itemize}
\item[] First note that $1-x = (1+x) -2x \in N_K(c) \cap -2N_K(-c) = \{1,c\}$. If $1-x \sim c$, then $1+2x = (1+x) +x = (1-x) + 3x \in N_K(-2c) \cap -2N_K(2c) \cap cN_K(-2) = \{-2\}$, so $1+2x \sim -2$. But then $1+3x = (1+x) + 2x = (1+2x) + x = (1-x) + 4x \in N_K(-2c) \cap 2N_K(c) \cap -2N_K(2c) \cap cN_K(-1) = \varnothing$: contradiction. Hence it must be that $1-x \sim 1$. \\
\end{itemize}

Next we show the claim about $y$.\\
\begin{itemize}
\item[] We have $1-y = (1+y)-2y \in N_K(-2c) \cap -N_K(c) = \{1,-c\}$. Suppose $1-y \sim -c$. Then we find $1+2y = (1+y)+y = (1-y)+3y \in N_K(c) \cap -N_K(-2c) \cap -cN_K(-1) = \{-1\}$. Also, $1-2y = (1+y)-3y = (1-y)-y = (1+2y)-4y \in N_K(-c) \cap -N_K(c) \cap -cN_K(2) \cap -N_K(2c) = \varnothing$: contradiction. Hence $1-y \sim 1$.\\
\end{itemize}

\noindent The proof is complete.
\end{proof}

\begin{lemma}
Assume $3 \sim 1$. Let $x \in K$ be such that $x \sim c, 1+x \sim -2$, and $y \in K$ be such that $y \sim -2c, 1+y \sim -1$. Then it is necessarily the case that $1-x \sim 1, 1-y \sim 1$.
\end{lemma}
\begin{proof}
Let us first show the claim about $x$.\\
\begin{itemize}
\item[] Firstly, $1-x  = (1+x) - 2x \in N_K(c) \cap -2N_K(-c) = \{1,c\}$. Assume that $1-x \sim c$. Then $1+2x = (1+x)+x = (1-x)+3x \in N_K(-2c) \cap -2N_K(2c) \cap cN_K(-1) = \varnothing$: contradiction. Therefore $1-x \sim 1$.\\
\end{itemize}

Next we show the claim about $y$.\\
\begin{itemize}
\item[] Firstly, $1-y = (1+y)-2y \in N_K(-2c) \cap -N_K(c) = \{1,-c\}$. Suppose $1-y \sim -c$. Then we find $1+2y = (1+y)+y = (1-y)+3y \in N_K(c) \cap -N_K(-2c) \cap -cN_K(-2) = \varnothing$: contradiction. Hence $1-y \sim 1$.\\
\end{itemize}

\noindent The proof is complete.
\end{proof}

We now rule out sub-case $B_1$. Concretely, we will show that we cannot have both $N_K(-c) = \langle -2, c \rangle$ \emph{and} $N_K(2c) = \langle -1, 2c \rangle$. Therefore, one of these two norm-groups must have dimension 1, contradicting the fact that $K$ is Demushkin.

\begin{proof}(Proposition \ref{caseb} $B_1$.)

We have that $-2 \in N_K(-c)$, so, after dividing by suitable squares, we may pick an element $x \in K$ with $x \sim c$ and $1+x  \sim -2$. Also, since $-1 \in N_K(2c)$, we may similarly pick an element $y \in K$ with $y \sim -2c, 1+y \sim -1$. By Lemmas C1 and C2, we know then, irrespective of the square value of $3$, that $1-x \sim 1, 1-y \sim 1$.

Now consider $1-xy$. Note that we have, as in Case A, the canonical decompositions
\begin{equation}
\label{decomp3}
1-xy = (1+ay)\left(1+(1+a^{-1}x)\frac{-ay}{1+ay}\right) 
\end{equation}
for any $a \in K$. If we define
\begin{equation}
G_a(x,y) = (1+ay)N( a(1+a^{-1}x)y(1+ay) ) \nonumber
\end{equation}
then $1-xy \in G_a(x,y)$ by equation (\ref{decomp3}). Similarly, $1-xy \in G_a(y,x)$. We also trivially have $1-xy \in N_K(xy) = N(-2)$. We now compute, using the results of Lemmas C1 and C2, that
\begin{eqnarray}
G_1(x,y) &=& -1 \cdot N_K(-2 \cdot (-2c) \cdot (-1)) = -N_K(-c) \nonumber  \\ \nonumber 
G_{-1}(x,y) &=& 1 \cdot N_K(-1 \cdot 1 \cdot (-2c) \cdot 1) = N_K(2c). \nonumber 
\end{eqnarray}
Hence
\begin{equation}
1-xy \in N_K(-2) \cap -N_K(-c) \cap N_K(2c) = \varnothing, \nonumber
\end{equation}
which gives the required contradiction.
\end{proof}

\subsection{Case $B_2$.}
\label{sec:c2}

In this case we have that $3$ is independent of $\pm 1, \pm 2$ in $k$, $N_k(2) = \langle 2,3 \rangle$, and $k^{\times}/(k^{\times})^2 = \langle -1,2,3 \rangle$. The proof of Lemma \ref{keylemma} shows that Case $B_2$ arises when $N_k(-1)$ has dimension 1, since otherwise the existence of the ordering on $k$ forces the equality $N_k(-2)=\langle 2 \rangle$ and thus $3 \sim 1$ or $3 \sim 2$. Since $2 \in N_k(-1)$, it follows that in Case $B_2$ that $N_k(-1) = \langle 2 \rangle$. Additionally, we always have the equality $N_k(2) = \langle -1,2 \rangle$. This follows from the fact that the corresponding equality has to hold in $K$, since $K$ is Demushkin and the right-hand-side is trivially contained in $N_K(2)$. The norm groups $N_k(a)$ for $a=\pm 3, \pm 6$ can be easily deduced from this, using reciprocity and the fact that $N_k(-3)$ and $N_k(-6)$ must contain only positive integers (modulo squares), due to the ordering on $k$. We get
\begin{eqnarray}
    N_k(-1) &=& \langle 2 \rangle \nonumber \\
    N_k(2) &=&\langle -1,2 \rangle \nonumber \\
    N_k(-2)& =& \langle 2, 3 \rangle \nonumber \\
    N_k(3) &=& \langle -2,-3 \rangle \nonumber \\
    N_k(-3) &=& \langle 3 \rangle \nonumber \\
    N_k(6) &=& \langle -2, 3 \rangle \nonumber \\
    N_k(-6) &=& \langle 6 \rangle. \nonumber
\end{eqnarray}

Using this lattice it is easy to deduce the following
\begin{lemma}
\label{finallemma}
In sub-case $B_2$, $5 \sim 2$ in $k$. 
\end{lemma}
\begin{proof}
We have $5=1+4=2+3 \in N_k(-1) \cap 2N_k(-6) = \{2\}$. 
\end{proof}

Next, since $k$ is relatively algebraically closed in $K$, $N_k(-2) = \langle 2, 3 \rangle$ generates a subspace of $N_K(-2)$, and being of the same dimension they are thereby equal. By reference to the lattice presented in section \ref{sec:c0}, we know $N_K(-2) = \langle 2, -c \rangle$, implying that $3 \sim -c$ or $3 \sim -2c$. Since the lattice in section \ref{sec:c0} is invariant when replacing $c$ by $2c$, we may without loss of generality take $c=-3$. This gives us the following lattice for $K$:
\begin{eqnarray}
    N_K(-1) &=& \langle 2,-3 \rangle \nonumber \\
    N_K(2) &=&\langle -1,2 \rangle \nonumber \\
    N_K(-2)& =& \langle 2, 3 \rangle \nonumber \\
    N_K(-3) &=& \langle -1,3 \rangle \nonumber \\
    N_K(3) &=& \langle -2,-3 \rangle \nonumber \\
    N_K(-6) &=& \langle -1, 6 \rangle \nonumber \\
    N_K(6) &=& \langle -2,3 \rangle. \nonumber
\end{eqnarray}

The fact that $N_K(-6) = \langle -1,6 \rangle$ implies that there exists an $x \in K$ with $x \sim 6, 1+x \sim -1$. On the other hand, we prove that

\begin{lemma}
\label{contradiction}
There can be no element $x \sim 6$ with $1+x \sim -1$ in $K$, under the assumptions of sub-case $B_2$, provided that $5 \sim 2$ in $K$. 
\end{lemma}
\begin{proof}
Indeed, consider the element $1-3x$. We have
\begin{equation}
    1-3x = (1+x)-4x \in N_K(3x) \cap -N_K(-4x) = N_K(2) \cap -N_K(-6) = \{1,-1\}.
\end{equation}
Assume that $1-3x \sim 1$. Then
\begin{eqnarray}
    1-2x &=& (1+x)-3x = (1-3x)+x \nonumber \\
    \in N_K(2x) &\cap& -N_K(-3x) \cap N_K(-x) \nonumber \\
         = N_K(3) &\cap& -N_K(-2) \cap N_K(-6) = \varnothing,
\end{eqnarray}
giving a contradiction. So it must be that $1-3x \sim -1$. In this case,
\begin{equation}
    1-2x \in N_K(2x) \cap -N_K(-3x) \cap -N_K(x) = \{-3\},
\end{equation}
using the same decompositions as above. But, finally, this implies that
\begin{eqnarray}
    1-5x &=& (1+x)-6x = (1-3x)-2x = (1-2x)-3x \nonumber \\
    \in N_K(2x) &\cap& -N_K(-6x) \cap -N_K(-2x) \cap -3N_K(-x) \nonumber \\
    = N_K(3) &\cap& -N_K(-1) \cap -N_K(-3) \cap -3N_K(-6) = \varnothing,
\end{eqnarray}
a contradiction which completes the proof of the lemma.


\end{proof}

To finish the argument, $5 \sim 2$ in $k$ implies $5/2 \in k^2 \subset K^2$, and hence $5 \sim 2$ in $K$. Thus Lemmas \ref{finallemma} and \ref{contradiction} generate a contradiction, which rules out sub-case $B_2$.

\newpage

\bibliographystyle{amsalpha}
\bibliography{Demushkin_references}

\affiliationone{
   Jochen Koenigsmann\\
   Department of Mathematics,
   University of Oxford,  
   Andrew Wiles Building,
   Oxford OX2 6GG, 
   United Kingdom
   \email{jochen.koenigsmann@maths.ox.ac.uk\\
   }}
\affiliationtwo{
   Kristian Strommen\\
   Department of Physics,
   University of Oxford,  
   Clarendon Laboratory,
   Oxford OX1 3PU, 
   United Kingdom
   \email{kristian.strommen@physics.ox.ac.uk\\
   }}

\end{document}